\documentclass[hidelinks,onefignum,onetabnum]{siamart250211}

\usepackage{lipsum}
\usepackage{amsfonts}
\usepackage{graphicx}
\usepackage{epstopdf}
\usepackage{algorithmic}

\usepackage{amsmath}
\usepackage{graphicx}
\usepackage{epstopdf}
\usepackage{algorithmic}

\usepackage{bm}
\usepackage[ntheorem]{empheq}
\usepackage{enumitem}
\usepackage{multirow}
\usepackage{booktabs}
\ifpdf
  \DeclareGraphicsExtensions{.eps,.pdf,.png,.jpg}
\else
  \DeclareGraphicsExtensions{.eps}
\fi

\newsiamremark{remark}{Remark}
\newsiamremark{hypothesis}{Hypothesis}
\crefname{hypothesis}{Hypothesis}{Hypotheses}
\newsiamthm{claim}{Claim}
\newsiamremark{fact}{Fact}
\crefname{fact}{Fact}{Facts}

\headers{An Efficient ADMM Method for Ratio-Type Minimization}{Lang Yu, Nanjing Huang}

\title{An Efficient ADMM Method for Ratio-Type Nonconvex and Nonsmooth  Minimization in Sparse Recovery\thanks{Submitted to the editors DATE.
\funding{This work was funded by the National Natural Science Foundation of China (12171339, 12471296).}}}

\author{Lang Yu\thanks{Department of Mathematics, Sichuan University, Chengdu, 610064, China 
  (\email{nanjinghuang@hotmail.com; njhuang@scu.edu.cn}, ).}
\and Nanjing Huang\footnotemark[2]}

\usepackage{amsopn}

\ifpdf
\hypersetup{
  pdftitle={An Efficient ADMM Method for Ratio-Type Nonconvex and Nonsmooth  Minimization in Sparse Recovery},
  pdfauthor={Lang Yu, Nanjing Huang}
}
\fi

\begin{document}

\maketitle

\begin{abstract}
Sparse signal recovery based on nonconvex and nonsmooth optimization problems has significant applications and demonstrates superior performance in signal processing and machine learning. This work deals with a scale-invariant $\ell_{1/2}/\ell_{2}$ sparse minimization with nonconvex, nonseparable, ratio-type regularization to enhance the accuracy and stability of sparse recovery. Within the framework of the null space property, we analyze the conditions for exact and stable recovery in constrained minimization problem. For the unconstrained regularized minimization problem, we develop an alternating direction method of multipliers (ADMM) based on a splitting strategy and rigorously analyze its global convergence and linear convergence rate under reasonable assumptions. Numerical experiments demonstrate that the proposed method consistently outperforms existing approaches across diverse noise levels and measurement settings. Furthermore, experiments on neural network sparsity and generalization performance demonstrate that the method effectively improves prediction accuracy.
\end{abstract}



\section{Introduction}\label{sec:intro}
Within the framework of compressed sensing (CS), the sparse signal recovery problem is typically formulated as solving an underdetermined linear system or an $\ell_{0}$ norm minimization problem under noisy observations, or equivalently, by incorporating an $\ell_{0}$ regularization term into the objective function as a constraint. Specifically, one aims to solve the following problems:
\begin{subequations}
	\begin{align}
		&\min_{\bm{x} \in \mathbb{R}^n} \|\bm{x}\|_0 \quad \text{s.t.} \; \bm{Ax} = \bm{b} \; \text{or} \; \|\bm{Ax} - \bm{b}\|_2 \leq \epsilon_0, \label{eq:L0Con}   \\
		&\min_{\bm{x} \in \mathbb{R}^n} \frac{1}{2}\|\bm{Ax} - \bm{b}\|_2^2 + \lambda_0 \|\bm{x}\|_0, \label{eq:L0UnCon}
	\end{align}
\end{subequations}
where $\bm{A} \in \mathbb{R}^{m \times n}$ $(m \ll n)$ denotes the sensing matrix, $\bm{b} \in \mathbb{R}^m$ is the observation vector, $\epsilon_0$ specifies the noise level, and $\|\bm{x}\|_0$ represents the number of nonzero entries in $\bm{x}$, referred to as the $\ell_{0}$ norm. The regularization parameter $\lambda_0 > 0$ balances the trade-off between the fidelity term and the sparsity-promoting regularization term. However, both problems \eqref{eq:L0Con} and \eqref{eq:L0UnCon} are NP-hard and thus notoriously challenging to solve.  

Classical CS theory demonstrates that, when a signal is sparse under a certain basis and the sensing matrix satisfies the restricted isometry property (RIP) \cite{candes2006robust} or similar conditions, exact signal recovery can be achieved via convex optimization, in particular through $\ell_{1}$ norm minimization \cite{donoho2006compressed}. Nevertheless, when the sensing matrix violates RIP or exhibits high coherence, $\ell_{1}$ minimization may fail to identify the sparsest solution, often yielding suboptimal recovery results \cite{cherni2020spoq}. Candès et al. observed that $\ell_{1}$ minimization can produce solutions that are insufficiently sparse, limiting its applicability in challenging scenarios \cite{candes2006robust}. Moreover, for sparse signals with high dynamic range, the $\ell_{1}$ norm may excessively penalize large coefficients while underestimating smaller ones, thereby degrading recovery performance \cite{candes2008enhancing}.  

To address the inherent limitations of convex models, numerous nonconvex approaches have been proposed to construct sparsity-promoting functions that more closely approximate the $\ell_{0}$ norm. Among them, the $\ell_{p}$ quasi-norm ($0 < p < 1$) has demonstrated significant advantages over the $\ell_{1}$ norm, achieving sparser and more accurate solutions under weaker conditions, and enabling exact recovery \cite{chen2014convergence,huang2003unified,wang2021nonconvex}. Other notable nonconvex regularizers include the smoothly clipped absolute deviation (SCAD) \cite{fan2001variable}, fractional $\ell_{1}$ \cite{ lv2009unified}, transformed $\ell_{1}$ \cite{nikolova2000local}, log-sum penalty \cite{zhou2023iterative}, and minimax concave penalty (MCP) \cite{zhang2010nearly}. Many of these models possess separable structures, allowing for parallel or closed-form updates during optimization, thus achieving both high recovery accuracy and computational efficiency.  

Although separable structures facilitate efficient computation, certain models, particularly those that do not possess closed-form thresholding operators, must rely on more elaborate numerical procedures within each iteration, such as root-finding methods or nested sub-iterations. This requirement substantially increases the computational cost per iteration \cite{chartrand2012nonconvex,daubechies2004iterative,wipf2010iterative}. To alleviate these drawbacks, researchers have developed efficient nonconvex sparse optimization models with non-separable structures. Among these, $\ell_{1} - \ell_{2}$ \cite{esser2013method,yin2015minimization} has attracted considerable attention for its strong performance under high coherence, and can be effectively solved using the difference-of-convex algorithm (DCA) without additional smoothing or regularization \cite{lou2018fast,lou2015computing}. Yin et al. \cite{yin2014ratio} analyzed exact recovery conditions for $\ell_{1} - \ell_{2}$. Another important development is the $\ell_{1} / \ell_{2}$ minimization, which enjoys scale invariance and adapts well to high-dynamic-range signals \cite{wang2021limited,zeng2021analysis}. Originally introduced by Hoyer \cite{hoyer2002non} in the context of nonnegative matrix factorization, $\ell_{1} / \ell_{2}$ has since been recognized as an effective sparsity-promoting functional \cite{xu2021analysis}. Its scale invariance makes it a compelling candidate for approximating the scale-invariant $\ell_{0}$ norm \cite{rahimi2019scale}. Theoretically, Yin et al. \cite{yin2014ratio} proved that, for nonnegative signals, minimizing $\ell_{1} / \ell_{2}$ is equivalent to minimizing $\ell_{0}$, while Rahimi et al. \cite{rahimi2019scale} showed that an $s$-sparse solution is a local minimizer of $\ell_{1} / \ell_{2}$ under the strong null space property (sNSP). Tao \cite{tao2022minimization,tao2023study} derived the analytical form of its proximal operator, providing new algorithmic framework. Further extensions include $\ell_{1} / \ell_{\infty}$ \cite{wang2023variant}, $\ell_{1} - \alpha \ell_{2}$ \cite{ge2021dantzig}, $\alpha \ell_{1} - \beta \ell_{2}$ \cite{ding2019alphaL1}, and $q$-ratio models \cite{zhou2021minimization}, with successful applications in the blind deconvolution \cite{repetti2014euclid}, limited-angle computed tomography (CT) reconstruction \cite{wang2021limited}, sparse portfolio optimization \cite{wu2024sparse},  and image gradient sparsification \cite{wang2022minimizing}.  

Nowadays, from the classical $\ell_{1}$ minimization to more advanced nonconvex approaches such as $\ell_{1}\!-\!\ell_{2}$ and $\ell_{1}/\ell_{2}$, the field of sparse  recovery have witnessed substantial progress, as mentioned above. One important research direction is to address the design of regularization terms, which can provide the stronger representational capacity. A parallel line of work aims at developing optimization algorithms endowed with rigorous convergence guarantees and improved scalability. Motivated by the recent advances mentioned above, we introduce a ratio-based sparse optimization problem designed to enhance sparse recovery performance in challenging scenarios. In particular, we consider the minimization of the ratio between the $\ell_{1/2}$ quasi-norm and the $\ell_{2}$ norm:
\begin{subequations}
	\begin{align}
		\min_{\bm{x} \in \mathbb{R}^n} \mathcal{H}(\bm{x}) := \zeta \frac{\|\bm{x}\|_{1/2}^{1/2}}{\|\bm{x}\|_2^{1/2}} \quad \text{s.t.} \quad \bm{Ax} = \bm{b}, \label{eq:OptimizatopnObjectCon}   \\
		\min_{\bm{x} \in \mathbb{R}^n} \mathcal{H}(\bm{x}) := \zeta \frac{\|\bm{x}\|_{1/2}^{1/2}}{\|\bm{x}\|_2^{1/2}} + \frac{1}{2} \|\bm{Ax} - \bm{b}\|_2^2,\label{eq:OptimizatopnObjectUnCon}
	\end{align}
\end{subequations}
with the convention $\|\bm{0}\|_{1/2} / \|\bm{0}\|_2 = 1$ when $\bm{x} = \bm{0}$. Our focus lies in the theoretical recovery properties of the constrained minimization problem \cref{eq:OptimizatopnObjectCon} and the algorithmic design for solving the unconstrained minimization problem \cref{eq:OptimizatopnObjectUnCon}. Notably, $\ell_{1/2} / \ell_{2}$ retains scale-invariance. Compared with $\ell_{1} - \ell_{2}$, it better approximates the $\ell_{0}$ norm because $\ell_{1/2}$ is closer to $\ell_{0}$ than $\ell_{1}$ or $\ell_{p}$ for most $p > 1/2$, thereby suppressing small coefficients more aggressively and yielding sparser solutions. When $p \to 0$, $\ell_{p}$ tends to $\ell_{0}$, offering the strongest sparsity but with high nonconvexity, making optimization more challenging and susceptible to suboptimal local minima. For $0 < p < 1/2$, the sparsity is stronger than $\ell_{1/2}$ but at the expense of increased nonconvexity, reduced stability, and greater sensitivity to initialization and algorithm robustness. Conversely, for $1/2 < p < 1$, nonconvexity is weaker and stability improves, but sparsity is reduced and suppression of small coefficients becomes less pronounced; $p=1$ recovers the $\ell_{1}$ norm. In scenarios involving highly coherent matrices or high-dynamic-range signals, $\ell_{1/2}$ often preserves high reconstruction accuracy, avoids over-penalizing large coefficients, and admits analytic half-thresholding operators in many cases, enabling efficient implementation and rapid convergence \cite{xu2012l_,zong2012re,zeng2014l}. \cref{fig:Contourfig} illustrates the contour plots of various sparsity measures. It is evident that, near the origin, the contours of the $\ell_{1/2} / \ell_{2}$ sparsity-inducing function lie closer to the coordinate axes, thereby encouraging exact zeros and promoting stronger sparsity.
\vspace{-15pt} %
\begin{figure}[!htbp]
	\centering
	\setlength{\abovecaptionskip}{-1pt}   %
	\setlength{\belowcaptionskip}{0pt}   %
	\includegraphics[scale=0.25]{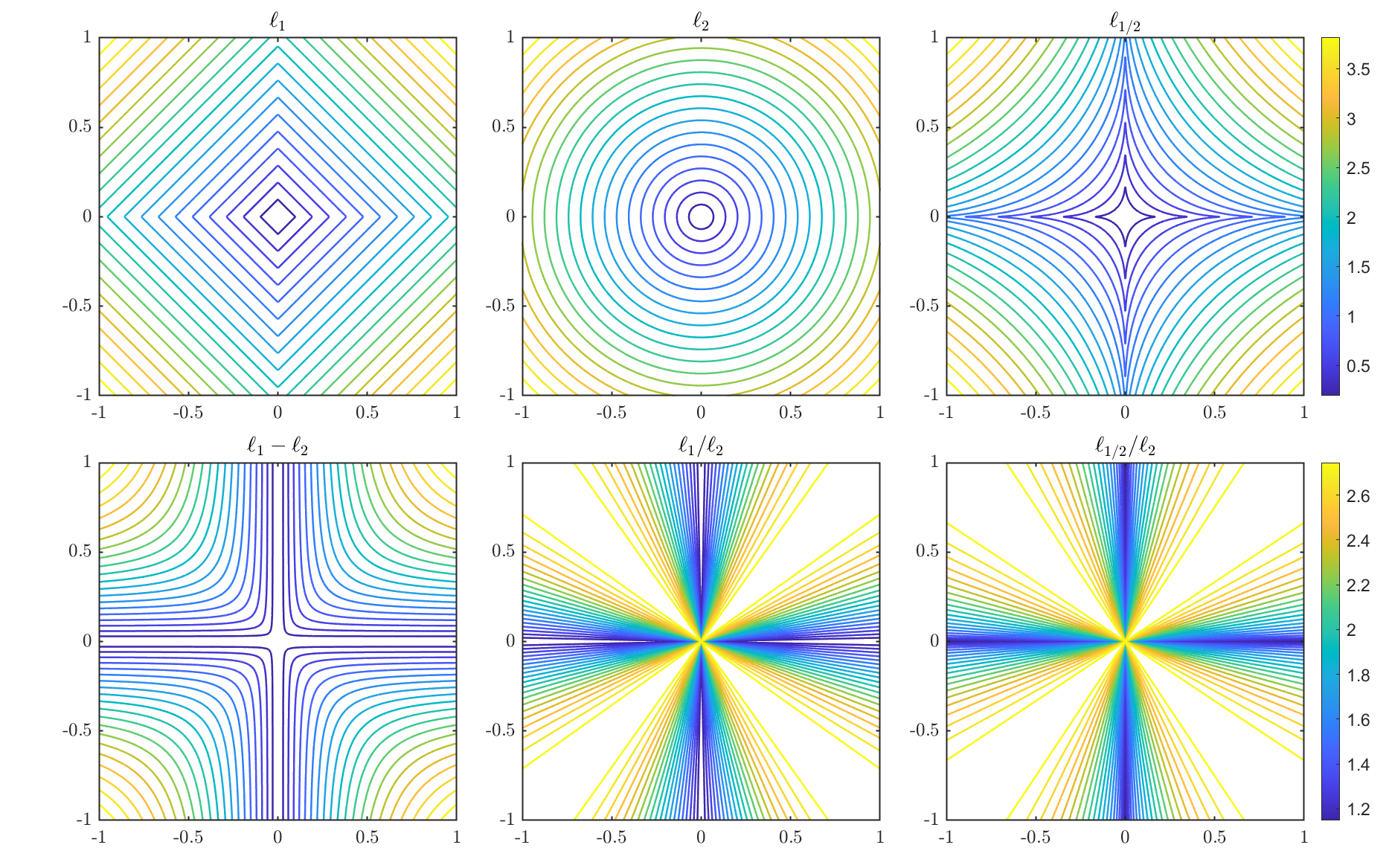} 
	\caption{Contour plots of different sparsity-inducing functions.}
	\label{fig:Contourfig}
\end{figure}
\vspace{5pt} %

The main contributions of the current work can be summarized as follows:
\begin{enumerate}[label=(\roman*)]
	\item We propose a novel, strongly sparsity-promoting, and scale-invariant $\ell_{1/2} / \ell_{2}$ sparse optimization problem.
	\item We establish exact and stable recovery conditions for the constrained $\ell_{1/2} / \ell_{2}$ minimization problem.
	\item We design an efficient splitting framework and develop an ADMM-based algorithm to solve the proposed problem.
	\item We obtain the global convergence and  linear convergence rate for the proposed ADMM algorithm.
\end{enumerate}

The rest of this work is structured as follows: Section \ref{sec:notation} introduces relevant preliminaries and notation. Section \ref{sec:analysis} presents the theoretical analysis of the $\ell_{1/2} / \ell_{2}$ model. Section \ref{sec:alg} describes the ADMM algorithm design. Section \ref{sec:convergence} provides the convergence analysis. Section \ref{sec:experiments} reports numerical experiment results. Finally, Section \ref{sec:conclusion} concludes the work.

\section{Notation and preliminaries}\label{sec:notation}

We begin by introducing the notation and terminology that will be used consistently throughout the work. Let $\mathbb{R}^n$ denote the $n$-dimensional Euclidean space. Lowercase boldface letters (e.g., $\bm{x}$) represent vectors, where $x_i$ denotes the $i$-th entry of $\bm{x}$. Uppercase boldface letters (e.g., $\bm{A}$) denote matrices. For a vector $\bm{x} \in \mathbb{R}^n$, the $\ell_p$ norm for $p \ge 1$ is defined as $\|\bm{x}\|_p = \left( \sum_{i=1}^{n} |x_i|^p \right)^{1/p}$. For $0 < p < 1$, the quantity $\|\bm{x}\|_p$ defines a nonconvex but useful sparsity-promoting functional, often referred to as the $\ell_p$ quasi-norm. The $\ell_0$ norm of a vector $\bm{x}$, denoted $\|\bm{x}\|_0$, is defined as the number of nonzero entries in $\bm{x}$. Let $[n] := \{i \in \mathbb{Z}_{>0} \mid 1 \le i \le n\}$ denote the standard index set of cardinality $n$. Define $[n]_s := \{T \subset [n] \mid |T| \le s\}$ as the collection of all subsets of $[n]$ whose cardinality does not exceed $s$, where $T$ denotes an index set and $|T|$ its cardinality. For any $T \in [n]_s$, we denote its complement by $T^c := [n] \setminus T$. We denote the support set of $\bm{x}$ by $\mathrm{supp}(\bm{x})= \{ i \in [n] : x_i \ne 0 \}$. For a given point $\bm{x} \in \mathbb{R}^n$ and radius $r > 0$, the closed Euclidean ball centered at $\bm{x}$ is denoted by $\mathbb{B}_r(\bm{x}) = \{ \bm{y} \in \mathbb{R}^n : \|\bm{y} - \bm{x}\|_2 \le r \}$.

Given a closed set $D \subseteq \mathbb{R}^n$, the indicator function $\delta_D : \mathbb{R}^n \to \{0, +\infty\}$ is defined by $\delta_D(\bm{x}) = 0$ if $\bm{x} \in D$, and $\delta_D(\bm{x}) = +\infty$ otherwise. The distance from a point $\bm{x} \in \mathbb{R}^n$ to a closed set $D \subseteq \mathbb{R}^n$ is defined as $\mathrm{dist}(\bm{x}, D) := \inf_{\bm{y} \in D} \|\bm{x} - \bm{y}\|$. We denote by $\bm{A}_T \in \mathbb{R}^{m \times |T|}$ the submatrix of matrix $\bm{A}\in \mathbb{R}^{m \times n}$ consisting of the columns whose indices belong to $T$. The matrix $\bm{I}$ is the identity matrix with dimensions conforming to the context. Let $\bm{\text{ker}(A)}$ and $R\bm{(A)}$ denote the null space and the range of matrix $\bm{A}$, respectively.

Let $f : \mathbb{R}^n \to (-\infty, +\infty]$ be an extended real-valued function. The function $f$ is called proper if its effective domain $ \operatorname{dom} f := \{ \bm{x} \in \mathbb{R}^n \mid f(\bm{x}) < +\infty \} $ is non-empty and $f(\bm{x}) > -\infty$ for all $\bm{x} \in \mathbb{R}^n$. It is lower semicontinuous (l.s.c.) at $\bm{x}$ if, for every sequence $\bm{x}^k \to \bm{x}$, one has $ \liminf_{k \to \infty} f(\bm{x}^k) \geq f(\bm{x}).$ If $f$ is l.s.c. at every point, we say that $f$ is l.s.c. A function is called closed if it is both proper and l.s.c. The sign function $\mathrm{sgn}:\mathbb{R}\to\{-1,0,1\}$ is defined as $\mathrm{sgn}(x)=1$ if $x>0$, $\mathrm{sgn}(x)=0$ if $x=0$, and $\mathrm{sgn}(x)=-1$ if $x<0$. It is often used in subdifferential expressions involving nonsmooth terms. Let $f: \mathbb{R}^n \to (-\infty, +\infty]$ be a proper closed function. The regular (Fréchet) subdifferential of $f$ at $\bm{x}$ is defined as
\begin{equation*}
	\widehat{\partial} f(\bm{x}) := \left\{ \bm{v} \in \mathbb{R}^n \,\middle|\, \liminf_{\bm{y} \to \bm{x},\, \bm{y} \ne \bm{x}} \frac{f(\bm{y}) - f(\bm{x}) - \langle \bm{v}, \bm{y} - \bm{x} \rangle}{\|\bm{y} - \bm{x}\|} \ge 0 \right\}.
\end{equation*}
The limiting subdifferential (also known as Mordukhovich subdifferential) of $f$ at $\bm{x} \in \operatorname{dom} f$ is defined by 
\begin{equation*}
	\partial f(\bm{x}) := \left\{ \bm{v} \in \mathbb{R}^n \,\middle|\, \exists\, \bm{x}^k \to \bm{x},\; f(\bm{x}^k) \to f(\bm{x}),\; \bm{v}^k \in \widehat{\partial} f(\bm{x}^k),\; \bm{v}^k \to \bm{v} \right\}.
\end{equation*}
A point $\bm{x}^\ast \in \operatorname{dom} f$ is called a critical point or stationary point if $\bm{0} \in \partial f(\bm{x}^\ast)$.

\section{The analysis of $\ell_{1/2}$ / $\ell_2$ model} \label{sec:analysis}
In this section, we first provide an illustrative example to highlight the superiority of the $\ell_{1/2}/\ell_2$ sparsity-inducing function. Then we proceed to analyze its theoretical properties.
\subsection{A Toy Example for Demonstration}
Consider a sensing matrix $ \bm{A} \in \mathbb{R}^{7 \times 8} $ and an observation vector $ \bm{b} \in \mathbb{R}^7 $, which are defined as
\begin{equation}
	\bm{A} = \begin{bmatrix}
		1 & -1 & 0 & 0 & 0 & 0 & 0 & 0 \\
		1 &  0 & -1 & 0 & 0 & 0 & 0 & 0 \\
		0 &  1 & 1 & 1 & 0 & 0 & 0 & 0 \\
		-2 & -2 & 0 & 0 & 1 & 0 & 0 & 0 \\
		-1 & -1 & 0 & 0 & 0 & 1 & 0 & 0 \\
		-1 &  0 & -1 & 0 & 0 & 0 & 1 & 0 \\
		-2 & -2 & -2 & 0 & 0 & 0 & 0 & 1
	\end{bmatrix}
\end{equation}
and $\bm{b} = \left[0, 0, 20, 40, 16, 25, 39\right]^{T}$. It is easy to verify that all solutions to the linear system $\bm{A}\bm{x} = \bm{b}$ can be expressed in the parametric form $\bm{x}(\sigma) = (\sigma,\ \sigma,\ \sigma,\ 20 - 2\sigma,\ 40 + 4\sigma,\ 16 + 2\sigma,\ 25 + 2\sigma,\ 39 + 2\sigma)^T, \sigma \in \mathbb{R}.$ The sparsest solution is attained at $\sigma = 0$, where the vector $\bm{x}$ contains only 5 nonzero entries. In contrast, several other values of $\sigma$, such as $\sigma = \pm 10$, $-8$, $-12.5$, and $19.5$, yield solutions with higher sparsity levels, namely 7 nonzero entries. These correspond to local minima under various nonconvex sparsity-inducing functions but they are suboptimal with respect to sparsity.

\vspace{-5pt} %
\begin{figure}[!htbp]
	\centering
	\setlength{\abovecaptionskip}{0pt}  
	\includegraphics[scale=0.30]{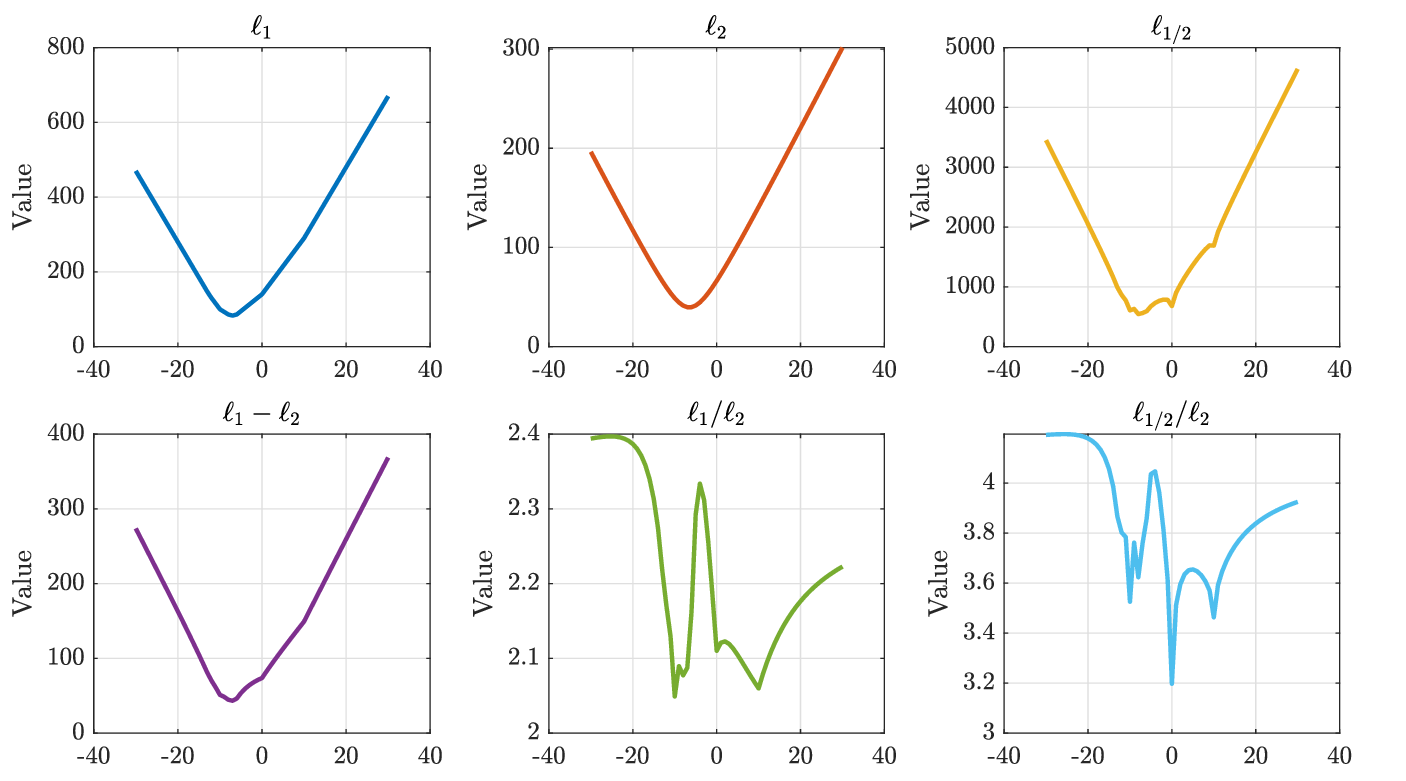} 
	\caption{The objective functions of the toy example illustrate that only $\ell_{1/2}/\ell_2$ can find $\sigma = 0$ as the global minimizer.}
	\label{fig:toyexampfig}
\end{figure}
\vspace{-3pt} %
In \cref{fig:toyexampfig}, we illustrate the behavior of various sparsity-inducing functions with respect to the parameter $\sigma$. Among all models, only $\ell_{1/2}/\ell_2$ successfully identifies the globally sparsest solution at $\sigma = 0$. These results highlight the superior sparsity-promoting capability of $\ell_{1/2}/\ell_2$ under this toy setting.

\subsection{Exact and stable recovery}
We first give the following lemma, which will be frequently invoked later.
\begin{lemma}[Ratio Bound] \label{lem:RatioBounds}
	For any nonzero vector $\bm{x} \in \mathbb{R}^n \setminus \{\bm{0}\}$, it holds that
	\begin{equation}
		1 \leq \frac{\|\bm{x}\|_{1/2}^{1/2}}{\|\bm{x}\|_2^{1/2}} =\frac{\sum_{i=1}^{n}|x_i|^{1/2}}{\left(\sum_{i=1}^{n}|x_i|^{2}\right)^{1/4}} \leq n^{3/4}.
	\end{equation}
\end{lemma}
\begin{proof}
	Invoking the Cauchy--Schwarz inequality, we have
	\begin{equation} \label{eq:CauchySchwarzL12}
		\left( \sum_{i=1}^n 1 \cdot |x_i|^{1/2} \right)^2 \le \left( \sum_{i=1}^n 1^2 \right) \left( \sum_{i=1}^n |x_i| \right) = n \|\bm{x}\|_1.
	\end{equation}
	By the definition of the $\ell_{1/2}$ quasi-norm and combining with inequality \cref{eq:CauchySchwarzL12}, it follows that $\|\bm{x}\|_{1/2} \le n \|\bm{x}\|_1.$ Since $\|\bm{x}\|_1 \le \|\bm{x}\|_{1/2}$ holds trivially from the definition of the $\ell_{1/2}$ quasi-norm, taking square roots on both sides of inequalities yields
	\begin{equation} \label{eq:CauchySchwarzL1}
		\|\bm{x}\|_1^{1/2} \le \|\bm{x}\|_{1/2}^{1/2} \le \sqrt{n} \|\bm{x}\|_1^{1/2}.
	\end{equation}
	Furthermore, applying the Cauchy--Schwarz inequality to $\|\bm{x}\|_1$ and $\|\bm{x}\|_2$, we obtain the classical bound $\|\bm{x}\|_2 \le \|\bm{x}\|_1 \le \sqrt{n} \|\bm{x}\|_2.$
	Substituting this into inequality \cref{eq:CauchySchwarzL1}, one has
	\begin{equation}
		\frac{\|\bm{x}\|_{1/2}^{1/2}}{\|\bm{x}\|_2^{1/2}} \le \frac{\sqrt{n} \cdot \|\bm{x}\|_1^{1/2}}{\|\bm{x}\|_2^{1/2}} \le \frac{\sqrt{n} \cdot n^{1/4} \|\bm{x}\|_2^{1/2}}{\|\bm{x}\|_2^{1/2}} = n^{3/4}.
	\end{equation}
	Similarly, using the inequality $\|\bm{x}\|_2 \le \|\bm{x}\|_1$, we can deduce the lower bound $\frac{\|\bm{x}\|_{1/2}^{1/2}}{\|\bm{x}\|_2^{1/2}} \ge \frac{\|\bm{x}\|_1^{1/2}}{\|\bm{x}\|_2^{1/2}} \ge \frac{\|\bm{x}\|_2^{1/2}}{\|\bm{x}\|_2^{1/2}} = 1.$ Combining both bounds, it holds that $ 1 \le \frac{\|\bm{x}\|_{1/2}^{1/2}}{\|\bm{x}\|_2^{1/2}} \le n^{3/4}$.
\end{proof}

\begin{definition}[Extended Null Space Property (eNSP)]\label{def:eNSP}
	Given a matrix $ \bm{A} \in \mathbb{R}^{m \times n}$ and parameters $p \in (0,1]$ and $c \in (0,1)$, we say that $\bm{A}$ satisfies the eNSP of order $ s $ if
	\begin{equation}
		(1-c)^{1/p}\|\bm{v}_T\|_p \leq c^{1/p}\|\bm{v}_{T^{c}}\|_p, \quad \forall \bm{v} \in \ker(\bm{A}) \setminus \{\bm{0}\}, \forall T \subset [n]_s.
	\end{equation}
\end{definition}

\begin{remark}
	The eNSP can be regarded as a natural generalization of the plain null space property (NSP) \cite{donoho2001uncertainty}. In particular, when $ c = 1/2 $ and $ p = 1 $, the eNSP reduces to the plain NSP. For $ c \in (0, 1/2) $ and $ p = 1 $, the eNSP is strictly stronger than the plain NSP in the sense that any matrix satisfying the eNSP also satisfies the NSP, but the converse does not necessarily hold. Moreover, when $ p = 1 $ and $ c = 1/((1 + s)^p + 1) $, the eNSP coincides with the strong null space property (sNSP) \cite{rahimi2019scale} of order $ s $.
\end{remark}

Next, we study the sparse recovery properties of the constrained $\ell_{1/2}/\ell_2$ minimization problem \cref{eq:OptimizatopnObjectCon}. Leveraging the eNSP, we derive a sufficient condition ensuring exact recovery of any $s$-sparse vector. This result offers a theoretical basis for analyzing nonconvex, scale-invariant regularization schemes.
\begin{theorem}\label{thm:local_minimality}
	For $0 < p \le 1$, if an $ m \times n $ matrix $ \bm{A} $ satisfies the eNSP with parameter pair $\left(s, \frac{1}{2}-\frac{1}{4}s^{-p/2+p^2/2}\right) $ and 
	$\inf_{\bm{v}\in\text{ker}(\bm{A})\setminus\{\bm{0}\}}\frac{\|\bm{v}\|_{p}}{\|\bm{v}\|_{2}} \ge 2^{1/p}s^{1/p-p/2}.$ Then for every $ \bm{v} \in \ker(\bm{A}) $ and any $ s $-sparse solution of $ \bm{A}\bm{x} = \bm{b} $ ($ \bm{b} \neq \bm{0} $), one has
	\begin{equation}
		\frac{\|\bm{x}\|_{p}^{p}}{\|\bm{x}\|_{2}^{p}} \leq \frac{\|\bm{x} + \bm{v}\|_{p}^{p}}{\|\bm{x} + \bm{v}\|_{2}^{p}}.
	\end{equation}
\end{theorem}

\begin{proof}
	It is evident that the conclusion holds when $\bm{v} = 0$. In the following, we examine the situation where $\bm{v} \neq \bm{0}$. Fix an $\bm{x}$ such that $|\text{supp}\left(\bm{x}\right)| \leq s$, and let $T \subset [n]_s$ denote the support of $\bm{x}$. It follows from $\|\bm{x} + \bm{v}\|_2 \leq \|\bm{x}\|_2 + \|\bm{v}\|_2$ that
	\begin{equation}
		\frac{\|\bm{x} + \bm{v}\|_p^p}{\|\bm{x} + \bm{v}\|_2^p}
		\geq \frac{\|(\bm{x} + \bm{v})_T\|_p^p + \|\bm{v}_{T^c}\|_p^p}{(\|\bm{x}\|_2 + \|\bm{v}\|_2)^p}.
	\end{equation}
	For $0 < p \le 1$ and any $\bm{x}, \bm{v} \in \mathbb{R}^n $, the $p$-triangle inequality $ \|\bm{x} + \bm{v}\|_p^p \leq \|\bm{x}\|_p^p + \|\bm{v}\|_p^p$ implies 
	\begin{equation}\label{eq:qtriangle}
		\left | x_i\right |^p = \left | x_i+v_i-v_i\right |^p \le \left |x_i+v_i\right |^p + \left |v_i\right |^p.     
	\end{equation}
	Accordingly, $\left | x_i\right |^p - \left |v_i\right |^p \le \left | x_i+v_i\right |^p $ holds. For all $i \in T$, by summing both sides of the inequality \cref{eq:qtriangle}, we obtain $\|\bm{x}\|_p^p - \|\bm{v}_{T}\|_p^p \leq \|\left(\bm{x} + \bm{v_{T}}\right)\|_p^p$. For $0 < p \le 1$, the inequality $(\|\bm{x}\|_2 + \|\bm{v}\|_2)^p \le \|\bm{x}\|_2^p + \|\bm{v}\|_2^p$ holds, since the function $\|\cdot\|^p$ is concave. Hence, for any $ \bm{v} \in \ker(\bm{A}) \setminus \{\bm{0}\} $, one has
	\begin{equation}
		\begin{aligned}
			&\frac{\|(\bm{x} + \bm{v})_T\|_p^p + \|\bm{v}_{T^c}\|_p^p}{(\|\bm{x}\|_2 + \|\bm{v}\|_2)^p}
			\geq \frac{\|\bm{x}\|_p^p + \|\bm{v}_{T^c}\|_p^p - \|\bm{v}_T\|_p^p}{\|\bm{x}\|_2^p + \|\bm{v}\|_2^p} \\[5pt]
			&\geq \min\left\{ \frac{\|\bm{x}\|_p^p}{\|\bm{x}\|_2^p}, \frac{\|\bm{v}_{T^c}\|_p^p - \|\bm{v}_T\|_p^p}{\|\bm{v}\|_2^p} \right\} \geq \min\left\{ \frac{\|\bm{x}\|_p^p}{\|\bm{x}\|_2^p}, \frac{\|\bm{v}\|_p^p - 2\|\bm{v}_T\|_p^p}{\|\bm{v}\|_2^p} \right\}.
		\end{aligned}
		\label{eq:concavity-inequality}
	\end{equation}
	Thus, we have 
	\begin{equation}
		\frac{\|\bm{x} + \bm{v}\|_p^p}{\|\bm{x} + \bm{v}\|_2^p}
		\geq \min\left\{\frac{\|\bm{x}\|_p^p}{\|\bm{x}\|_2^p}, \frac{\|\bm{v}\|_p^p - 2\|\bm{v}_T\|_p^p}{\|\bm{v}\|_2^p}\right\}
		\label{eq:concavity}
	\end{equation}
	and equality \cref{eq:concavity} holds if and only if $ 		\frac{\|\bm{x}\|_p^p}{\|\bm{x}\|_2^p} = \frac{\|\bm{v}\|_p^p - 2\|\bm{v}_T\|_p^p}{\|\bm{v}\|_2^p}.$
	For $0 < p < r < \infty$ and any vector $\bm{x} \in \mathbb{R}^n$, the inequality $\|\bm{x}\|_r \le \|\bm{x}\|_p \le n^{1/p - 1/r} \|\bm{x}\|_r$ holds. As a consequence, when $p<r=2$, one has $s^{1 - p/2} \geq \frac{\|\bm{x}\|_p^p}{\|\bm{x}\|_2^p}.$ According to the eNSP, it follows that $\|\bm{v}_T\|_p^p \leq c\|\bm{v}\|_p^p$. By setting $c = 1/2-s^{-p/2+p^2/2}/4$ in the eNSP condition, we have 
	\begin{equation}
		\frac{\|\bm{v}\|_p^p - 2\|\bm{v}_T\|_p^p}{\|\bm{v}\|_2^p} \ge \frac{\|\bm{v}\|_p^p - 2 \cdot \left(\frac{1}{2}-s^{-p/2+p^2/2}/4\right)\|\bm{v}\|_p^p}{\|\bm{v}\|_2^p} =  s^{-p/2+p^2/2}\frac{\|\bm{v}\|_p^p}{2\|\bm{v}\|_2^p}.
	\end{equation}
	Since $\inf_{\bm{v}\in ker(\bm{A})\setminus\{\bm{0}\}}\frac{\|\bm{v}\|_{p}}{\|\bm{v}\|_{2}} \ge 2^{1/p}s^{1/p-p/2}$, it follows that 
	\begin{equation}
		\frac{s^{-p/2+p^2/2}}{2}\frac{\|\bm{v}\|_p^p}{\|\bm{v}\|_2^p} \ge \frac{s^{-p/2+p^2/2}}{2} \cdot 2s^{1-p^2/2} = s^{1-p/2}.
	\end{equation}
	As desired, the inequality $\frac{\|\bm{x}\|_{p}^{p}}{\|\bm{x}\|_{2}^{p}} \leq \frac{\|\bm{x} + \bm{v}\|_{p}^{p}}{\|\bm{x} + \bm{v}\|_{2}^{p}}$ holds.
\end{proof}

Next, we analyze the robustness of $\ell_{1/2}/\ell_2$ minimization in the presence of noise. Consider $\bm{b} = \bm{A}\bm{x} + \bm{e},$
where $\bm{x}$ is sparse and $\|\bm{e}\|_2 \leq \tilde{\epsilon}$ with $\tilde{\epsilon}$ known a priori. In this scenario, the corresponding optimization problem can be formulated as
\begin{equation}\label{eq:cons_noise}
	\min_{\bm{x}} \;\frac{\|\bm{x}\|_p^p}{\|\bm{x}\|_2^p} 
	\quad \text{s.t.} \quad \|\bm{A}\bm{x}-\bm{b}\|_2 \leq \tilde{\epsilon}.
\end{equation}
Let $\bm{x}^*$ denote a minimizer of \cref{eq:cons_noise}. Then $\|\bm{A}\bm{x}^* - \bm{A}\bm{x}\|_2 \leq 2\tilde{\epsilon},$ and it follows that $ \frac{\|\bm{x}^*\|_p^p}{\|\bm{x}^*\|_2^p} \leq \frac{\|\bm{x}\|_p^p}{\|\bm{x}\|_2^p} \leq s^{1-p/2}.$

\begin{theorem}
	Let $\bm{x} \in \mathbb{R}^n$ be an $s$-sparse vector and $\bm{x}^*$ be a minimizer of \cref{eq:cons_noise}. Decompose the error as $\bm{x}^* - \bm{x} = \bm{u} + \bm{w}$ with $\langle \bm{u}, \bm{w} \rangle = 0$.
	If $\beta := 2^{1+p/2} 5^p s^{1-p/2} \frac{\|\bm{u}\|_2^p}{\|\bm{u}\|_p^p} < 1$, then for any $\alpha \in (\beta,1)$, the following statements hold:
	
	(i) If the inequalities
	\begin{subequations}\label{eq:cond} 
		\begin{align}
			\langle \bm{x}, \bm{x}^* \rangle &\geq \left(1 - \frac{7\alpha^2}{32}\right) \|\bm{x}\|_2 \|\bm{x}^*\|_2, \label{eq:conda} \\
			\|\bm{x}\|_2 &\leq \|\bm{x}^*\|_2 \leq \left(1 + \frac{\alpha}{4}\right) \|\bm{x}\|_2 \label{eq:condb}
		\end{align}
	\end{subequations}
	are satisfied, then
	\begin{equation}\label{eq:rubresult1}
		\|\bm{x}^* - \bm{x}\|_2 \leq \sqrt{\alpha}\, \|\bm{x}\|_2.
	\end{equation}
	
	(ii) For any $p\in (0,1]$ , if at least one of the conditions in \cref{eq:cond} is violated, then
	\begin{equation}\label{eq:UpErr}
		\|\bm{x}^* - \bm{x}\|_p \leq \Biggl[1 + \Bigl(\frac{\alpha^p}{\alpha^p - \beta}\Bigr)^{1/p}\Biggr]^{1/p} \|\bm{w}\|_p.
	\end{equation}
\end{theorem}

\begin{proof}
	(i) Assume all the conditions in \cref{eq:cond} hold. Then we have 
	\begin{align*}
		\|\bm{x}^* - \bm{x}\|_2^p 
		&= \left(\|\bm{x}^*\|^2_2 + \|\bm{x}_0\|^2_2 - 2\langle \bm{x}^*, \bm{x}_0 \rangle \right)^{p/2} \\
		&\leq\left[\left(\frac{1}{2}\alpha^2 + \frac{1}{2}\alpha\right) \|\bm{x}\|^2_2\right]^{p/2} \overset{a}{\leq} \left(\alpha \|\bm{x}\|^2_2\right)^{p/2}=\alpha^{p/2}\|\bm{x}\|_2^p,
	\end{align*}
	where the inequality $\overset{a}{\leq}$ uses $\alpha^2 < \alpha$ for $0 < \alpha < 1$. Taking the $p$-th root on both sides yields the desired inequality \cref{eq:rubresult1}. 
	
	(ii) We first show that
	\begin{equation} \label{eq:UpperBoundForm}
		\frac{\|\bm{x}^* - \bm{x}\|_p^p}{\|\bm{x}^* - \bm{x}\|_2^p} \leq \frac{2 \cdot 5^p}{\alpha^p}s^{1-p/2}.
	\end{equation}
	\textbf{Case 1:} Condition \cref{eq:conda} is violated. Then
	\begin{equation}\label{eq:convexRoubs1}
		\begin{aligned}
			\|\bm{x}^* - \bm{x}\|_2^p 
			&= \left(\|\bm{x}^*\|_2^2 + \|\bm{x}\|_2^2 - 2\langle \bm{x}, \bm{x}^* \rangle \right)^{p/2} \\
			& \overset{}{\geq} \left(\|\bm{x}^*\|_2^2 + \|\bm{x}\|_2^2 - 2\left(1-\frac{7\alpha^2}{32}\right)\|\bm{x}^*\|_2\|\bm{x}\|_2 \right)^{p/2}\\
			& = \left(\frac{7\alpha^2}{32}\left(\|\bm{x}^*\|_2^2 + \|\bm{x}\|_2^2\right)\right)^{p/2}.
		\end{aligned}
	\end{equation}
	Using the $p$-convexity inequality $y \mapsto y^{2/p}$, it follows that
	\begin{equation}\label{eq:convexRoubs2}
		\begin{aligned}
			\left(\frac{7\alpha^2}{32}\left(\|\bm{x}^*\|_2^2 + \|\bm{x}\|_2^2\right)\right)^{p/2} 
			& \overset{}{\geq} \left(\frac{7\alpha^2}{32}\left( \frac{\left(\|\bm{x}^*\|_2^p + \|\bm{x}\|_2^p\right)^{2/p}}{2^{2/p-1}}\right)\right)^{p/2}\\
			& = \left(\frac{7\alpha^2}{32}\right)^{p/2} \frac{\|\bm{x}^*\|_2^p + \|\bm{x}\|_2^p}{2^{1-p/2}}.
		\end{aligned}
	\end{equation}
	By combining inequalities \cref{eq:convexRoubs1} and \cref{eq:convexRoubs2}, we obtain
	\begin{equation}
		\begin{aligned}
			\frac{\|\bm{x}^* - \bm{x}\|_p^p}{\|\bm{x}^* - \bm{x}\|_2^p} 
			& \overset{}{\leq} \frac{\|\bm{x}^*\|_p^p + \|\bm{x}\|_p^p}{\|\bm{x}^* - \bm{x}\|_2^p} \leq \frac{\|\bm{x}^*\|_p^p + \|\bm{x}\|_p^p}{\left(\frac{7\alpha^2}{32}\right)^{p/2} \frac{\|\bm{x}^*\|_2^p + \|\bm{x}\|_2^p}{2^{1-p/2}}} \\
			& \leq \frac{2^{1+2p}}{\left(7\alpha^2\right)^{p/2}}\frac{\|\bm{x}\|_p^p}{\|\bm{x}\|_2^p} \leq \frac{2^{1+2p}}{\left(7\alpha^2\right)^{p/2}}s^{1-p/2} \overset{}{\leq} \frac{2 \cdot 5^p}{\alpha^p}s^{1-p/2}.
		\end{aligned}		
	\end{equation}
	\textbf{Case 2:} Condition \cref{eq:condb} is violated. Consider two subcases:
	
	(a) Suppose that $\|\bm{x}^*\|_p^p < \|\bm{x}\|_p^p$. Decompose $\bm{x}^* - \bm{x} = (\bm{x}^* - \bm{x})_T + \bm{x}^*_{T^c}$. Then
	\begin{equation}
		\begin{aligned}
			\left \|\bm{x} \right \|_p^p >\left \|\bm{x}^* \right \|_p^p 
			&= \left \|\bm{x} + \left(\bm{x}^* - \bm{x}\right) \right \|_p^p = \left \|\bm{x} + \left(\bm{x}^* - \bm{x}\right)_T \right \|_p^p + \left \|\bm{x}^*_{T^c} \right \|_p^p \\
			& \geq \left \|\bm{x} \right \|_p^p - \left \|\left(\bm{x}^* - \bm{x}\right)_T \right \|_p^p + \left \|\bm{x}^*_{T^c} \right \|_p^p
		\end{aligned}		
	\end{equation}
	and so $\left \|\bm{x}^*_{T^c} \right \|_p^p \leq \left \|\left(\bm{x}^* - \bm{x}\right)_T \right \|_p^p$. Moreover, the arithmetic mean -- geometric mean (AM–GM) inequality implies
	\begin{equation} \label{eq:Am-gmIneq}
		\begin{aligned}
			\frac{\|\bm{x}^* - \bm{x}\|_p^p}{\|\bm{x}^* - \bm{x}\|_2^p} &= \frac{\left \|\left(\bm{x}^* - \bm{x}\right)_T \right \|_p^p + \left \|\bm{x}^*_{T^c} \right \|_p^p}{\left(\|\bm{x}^* - \bm{x}\|_2^2\right)^{p/2}} = \frac{\left \|\left(\bm{x}^* - \bm{x}\right)_T \right \|_p^p + \left \|\bm{x}^*_{T^c} \right \|_p^p}{\left(\left \|\left(\bm{x}^* - \bm{x}\right)_T \right \|_2^2 + \left \|\bm{x}^*_{T^c} \right \|_2^2\right)^{p/2}} \\
			& \overset{}{\leq} 2^{p/2}\frac{\left \|\left(\bm{x}^* - \bm{x}\right)_T \right \|_p^p + \left \|\bm{x}^*_{T^c} \right \|_p^p}{\left(\left \|\left(\bm{x}^* - \bm{x}\right)_T \right \|_2 + \left \|\bm{x}^*_{T^c} \right \|_2\right)^{p}}\\
			& \overset{}{\leq} 2^{2-p/2}\frac{\left \|\left(\bm{x}^* - \bm{x}\right)_T \right \|_p^p }{\left \|\left(\bm{x}^* - \bm{x}\right)_T \right \|_2^p} \leq 2^{2-p/2}s^{1-p/2} \leq \frac{2 \cdot 5^p}{\alpha^p}s^{1-p/2}.
		\end{aligned}
	\end{equation}
	
	(b) Suppose that $\|\bm{x}^*\|_p^p \geq \|\bm{x}\|_p^p$ and $\|\bm{x}^*\|_2 > (1+\alpha/4)\|\bm{x}\|_2$. It follows from $\|\bm{x}^*\|_p^p \geq \|\bm{x}\|_p^p$ that $\|\bm{x}^*\|_2^p \geq \|\bm{x}\|_2^p$ and so
	\begin{equation*}
		\begin{aligned}
			\frac{\|\bm{x}^* - \bm{x}\|_p^p}{\|\bm{x}^* - \bm{x}\|_2^p} 
			&\leq \frac{\|\bm{x}^*\|_p^p + \|\bm{x}\|_p^p}{\left(\|\bm{x}^* \|_2 - \|\bm{x}\|_2 \right)^p} 
			\leq \left(\frac{4+\alpha}{\alpha}\right)^p \left(\frac{\|\bm{x}^*\|_p^p}{\|\bm{x}^*\|_2^p} + \frac{\|\bm{x}\|_p^p}{\|\bm{x}\|_2^p}\right) \\
			&\leq 2 \left(\frac{4+\alpha}{\alpha}\right)^p s^{1-p/2} 
			\leq \frac{2 \cdot 5^p}{\alpha^p}s^{1-p/2},
		\end{aligned}
	\end{equation*}
	where the last inequality uses $\alpha < 1$. Thus, whenever \cref{eq:cond} is violated, \cref{eq:UpperBoundForm} holds. 
	
	Next we show that \cref{eq:UpErr} holds. Since $\bm{x}-\bm{x^\ast} = \bm{u} + \bm{w}$ with $\langle \bm{u}, \bm{w}\rangle=0$, for any $p \in (0,1]$ with $\|\bm{u}\|_p \le \|\bm{w}\|_p$, one has
	\begin{equation*}
		\|\bm{x} - \bm{x}^*\|_p^p \leq \|\bm{u}\|_p^p + \|\bm{w}\|_p^p 
		\leq 2\|\bm{w}\|_p^p 
		\leq \frac{2-\beta}{1-\beta}\|\bm{w}\|_p^p 
		\leq \left(1 + \left(\frac{\alpha^p}{\alpha^p-\beta}\right)^{1/p}\right)\|\bm{w}\|_p^p
	\end{equation*}
	for all $\beta \in (0,1)$, which yields \cref{eq:UpErr}. Similarly, for any $p \in (0,1]$ with $\|\bm{u}\|_p > \|\bm{w}\|_p$, we have
	\begin{equation*}
		\frac{\|\bm{x}^* - \bm{x}\|_p^p}{\|\bm{x}^* - \bm{x}\|_2^p} 
		\geq \frac{\|\bm{u}\|_p^p - \|\bm{w}\|_p^p}{\left(\|\bm{u}\|_2^2 + \|\bm{w}\|_2^2\right)^{p/2}} 
		= \frac{1 - \frac{\|\bm{w}\|_p^p}{\|\bm{u}\|_p^p}}{\left(\sqrt{1 + \tfrac{\|\bm{w}\|_2^2}{\|\bm{u}\|_2^2}}\right)^{p/2}} \cdot \frac{\|\bm{u}\|_p^p}{\|\bm{u}\|_2^p} 
		\geq \pi(v)\frac{\|\bm{u}\|_p^p}{\|\bm{u}\|_2^p},
	\end{equation*}
	where $\pi(v) := \frac{1 - v^p}{\left(\sqrt{1 + v^2}\right)^{p/2}}$ and $v := \max\left\{\frac{\|\bm{w}\|_p}{\|\bm{u}\|_p}, \frac{\|\bm{w}\|_2}{\|\bm{u}\|_2}\right\}.$ If $\pi(v)\frac{\|\bm{u}\|_p^p}{\|\bm{u}\|_2^p} > 2\cdot\left(\frac{5}{\alpha}\right)^p s^{1-p/2}$, then inequality \cref{eq:UpErr} is violated, which implies that \cref{eq:cond} holds and so does \cref{eq:rubresult1}. Hence, it remains to analyze the case 
	\begin{equation*}
		\pi(v)\frac{\|\bm{u}\|_p^p}{\|\bm{u}\|_2^p} \leq 2\cdot\left(\frac{5}{\alpha}\right)^p s^{1-p/2}.
	\end{equation*}
	Since $\pi(v) \geq \tfrac{1}{2^{p/2}}(1 - v^p)$ for $v>0$, we obtain
	\begin{equation*}
		\frac{\beta}{\alpha^p 2^{p/2}} 
		= \frac{2 \cdot 5^p s^{1-p/2}}{\alpha^p}\frac{\|\bm{u}\|_2^p}{\|\bm{u}\|_p^p} 
		\geq \pi(v) 
		\geq \frac{1}{2^{p/2}}(1 - v^p),
	\end{equation*}
	which implies $v^p \geq 1 - \tfrac{\beta}{\alpha^p}$. Hence, one has $ \|\bm{u}\|_p^p = v^{-p}\|\bm{w}\|_p^p \leq (1 -\beta/\alpha^p)^{-1/p}\|\bm{w}\|_p^p.$ For $0<p<1$, the fact that $\|\bm{u}-\bm{w}\|_p^p \leq \|\bm{u}\|_p^p + \|\bm{w}\|_p^p$ shows that \cref{eq:UpErr} is true.
\end{proof}

\section{Algorithmic Scheme} \label{sec:alg}

In this section, we develop an ADMM iteration algorithm for $\ell_{1/2}/ \ell_2$ minimization problem and employ an inner ADMM scheme for the $\bm{x}$-subproblem.

\subsection{The $\ell_{1/2}/ \ell_2$ minimization via ADMM iteration}

We now present the minimization of the unconstrained $\ell_{1/2}/ \ell_2$ problem \cref{eq:OptimizatopnObjectUnCon} via the ADMM framework. 
With this aim, we introduce an auxiliary variable $\bm{y}$, which reformulates \cref{eq:OptimizatopnObjectUnCon} into the constrained optimization problem
\begin{equation}\label{eq:ConstrainedFormulation}
	\begin{aligned}
		& \min_{\bm{x}, \bm{y} \in \mathbb{R}^n} \quad \zeta \frac{\|\bm{x}\|_{1/2}^{1/2}}{\|\bm{x}\|_2^{1/2}} + g(\bm{y}), \\
		& \quad \text{s.t.} \quad \quad \bm{x} = \bm{y},
	\end{aligned}
\end{equation}
where $g(\bm{y}) := \tfrac{1}{2} \|\bm{A}\bm{y} - \bm{b}\|_2^2$. The augmented Lagrangian function of \cref{eq:ConstrainedFormulation} is given by
\begin{equation}\label{eq:AugmentedLagrangian}
	\mathcal{L}_{\rho}(\bm{x}, \bm{y}, \boldsymbol{\lambda}) 
	= \zeta \frac{\|\bm{x}\|_{1/2}^{1/2}}{\|\bm{x}\|_2^{1/2}} 
	+ g(\bm{y}) 
	+ \langle \boldsymbol{\lambda}, \bm{x} - \bm{y} \rangle 
	+ \frac{\rho}{2} \|\bm{x} - \bm{y}\|_2^2,
\end{equation}
where $\bm{\lambda}$ denotes the Lagrange multiplier and $\rho>0$ is the penalty parameter. By the augmented Lagrangian function \cref{eq:AugmentedLagrangian}, the ADMM iteration scheme is then expressed as
\begin{subequations}\label{eq:admm_updates}
	\begin{empheq}[left=\empheqlbrace]{align}
		& \bm{x}^{k+1} = \arg\min_{\bm{x} \in \mathbb{R}^n} \; 
		\zeta \frac{\|\bm{x}\|_{1/2}^{1/2}}{\|\bm{x}\|_2^{1/2}} 
		+ \frac{\rho}{2} \left\| \bm{x} - \bm{y}^k + \frac{\boldsymbol{\lambda}^k}{\rho} \right\|_2^2 , \label{eq:admm_x} \\
		& \bm{y}^{k+1} = \arg\min_{\bm{y} \in \mathbb{R}^n} \; 
		g(\bm{y}) + \frac{\rho}{2} \left\| \bm{y} - \bm{x}^{k+1} - \frac{\bm{\lambda}^k}{\rho} \right\|_2^2 , \label{eq:admm_y} \\
		& \bm{\lambda}^{k+1} = \bm{\lambda}^k + \rho (\bm{x}^{k+1} - \bm{y}^{k+1}). \label{eq:admm_z}
	\end{empheq}
\end{subequations}

The $\bm{y}$-subproblem in \cref{eq:admm_y} admits a closed-form solution
\begin{equation}\label{eq:y-updata}
	\bm{y}^{k+1} = \left( \bm{I}_n + \tfrac{1}{\rho} \bm{A}^\top \bm{A} \right)^{-1} 
	\left( \dfrac{\bm{A}^\top \bm{b}}{\rho} + \dfrac{\bm{\lambda}^k}{\rho} + \bm{x}^{k+1} \right).
\end{equation}
In addition to the Sherman--Morrison--Woodbury (SMW) formula, the linear system \cref{eq:y-updata} can also be solved iteratively via the conjugate gradient (CG) method, which is particularly suitable for large-scale or sparse matrices.

\subsection{Inner ADMM for the \texorpdfstring{$\bm{x}$}{x}-subproblem} 
To solve the $\bm{x}$-subproblem in \cref{eq:admm_x}, we employ an inner ADMM scheme, denoted as $\text{ADMM}_\text{inner}$. Specifically, introduce an auxiliary variable $\bm{u}$ and reformulate \cref{eq:admm_x} as
\begin{equation}\label{eq:admm_i}
	\begin{aligned}
		& \min_{\bm{x}, \bm{u} \in \mathbb{R}^n} \quad \zeta \frac{\|\bm{x}\|_{1/2}^{1/2}}{\|\bm{u}\|_2^{1/2}} + \frac{\rho}{2} \|\bm{x} - \bm{\theta}^k\|_2^2 \\
		& \;\; \text{s.t.} \quad \quad \quad  \bm{x} = \bm{u},
	\end{aligned}
\end{equation}
where $\bm{\theta}^k = \bm{y}^k - \bm{\lambda}^k / \rho$. The augmented Lagrangian function for \cref{eq:admm_i} is
\begin{equation}\label{eq:innerADMM}
	\mathcal{L}_{\gamma}^k(\bm{x}, \bm{u}, \bm{\vartheta}) = \zeta \frac{\|\bm{x}\|_{1/2}^{1/2}}{\|\bm{u}\|_2^{1/2}} + \frac{\rho}{2} \|\bm{x} - \bm{\theta}^k\|_2^2 + \langle \bm{\vartheta}, \bm{x} - \bm{u} \rangle + \frac{\gamma}{2} \|\bm{x} - \bm{u}\|_2^2,
\end{equation}
where $\bm{\vartheta}$ is the Lagrange multiplier and $\gamma > 0$ is the penalty parameter. The $\text{ADMM}_\text{inner}$ iteration framework for minimizing \cref{eq:admm_i} is given by
\begin{subequations}\label{eq:inneradmm_updates}
	\begin{empheq}[left=\empheqlbrace]{align}
		&\bm{x}_{t+1} = \arg\min_{\bm{x} \in \mathbb{R}^n} \; \zeta \frac{\|\bm{x}\|_{1/2}^{1/2}}{\|\bm{u}_t\|_2^{1/2}} + \frac{\rho}{2} \|\bm{x} - \bm{\theta}^k\|_2^2 + \frac{\gamma}{2} \|\bm{x} - \bm{u}_t + \frac{\bm{\vartheta}_t}{\gamma}\|_2^2 , \label{eq:inneradmm_x}\\
		&\bm{u}_{t+1} = \arg\min_{\bm{u} \in \mathbb{R}^n} \; \zeta \frac{\|\bm{x}_{t+1}\|_{1/2}^{1/2}}{\|\bm{u}\|_2^{1/2}} + \frac{\delta}{2} \|\bm{u} - \bm{x}_{t+1} - \frac{\bm{\vartheta}_t}{\gamma}\|_2^2 , \label{eq:inneradmm_u} \\
		&\bm{\vartheta}_{t+1} = \bm{\vartheta}_t + \gamma(\bm{x}_{t+1} - \bm{u}_{t+1}),\label{eq:inneradmm_v}
	\end{empheq}
\end{subequations}
where the subscript $t$ denotes the inner iteration index, in contrast to the superscript $k$ for the outer iterations.

Through algebraic manipulation, the subproblem \cref{eq:inneradmm_x} can be rewritten as
\begin{equation}\label{eq:xsubproblem}
	\begin{aligned}
		\bm{x}_{t+1} &= \arg\min_{\bm{x} \in \mathbb{R}^n} \; \zeta \frac{\|\bm{x}\|_{1/2}^{1/2}}{\|\bm{u}_t\|_2^{1/2}} + \frac{\rho}{2} \|\bm{x} - \bm{\theta}^k\|_2^2 + \frac{\gamma}{2} \|\bm{x} - \bm{u}_t + \frac{\bm{\vartheta}_t}{\gamma}\|_2^2 \\
		& =  \arg\min_{\bm{x} \in \mathbb{R}^n} \;\frac{\gamma + \rho}{2}\left(\|\bm{x}-\bm{m}_t\|_2^2 + \tilde{\delta}\|\bm{x}\|_{1/2}^{1/2}  \right) + C_{\gamma},
	\end{aligned}
\end{equation}
where $C_{\gamma} = \frac{\gamma\rho}{2\left(\gamma + \rho\right)}\|\bm{\theta}_k-\bm{p}_t\|_2^2$, $\tilde{\delta} = \frac{2\zeta}{\left(\gamma+\rho\right)\|\bm{u}_t\|_2^{1/2}}$, $\bm{p}_t = \bm{u}_t-\bm{\vartheta}_t/\gamma$, and $\bm{m}_t = \frac{\rho\bm{\theta}_k + \gamma \bm{p}_t}{\gamma + \rho}$. This objective is component-wise separable, allowing the $\bm{x}$-subproblem to be solved independently for each component $x_{t+1}^i$ via the scalar minimization
\begin{equation}\label{eq:xsubproblem_scalar}
	x_{t+1}^i = \arg\min_{x \in \mathbb{R}} \left\{ \tilde{\delta} |x|^{1/2} + \frac{1}{2}(x - m_t^i)^2 \right\},\quad \forall i.
\end{equation}
The closed-form solution \cite{cao2013fast} for \cref{eq:xsubproblem_scalar} is
\begin{equation} \label{eq:updata_subx}
	x_{t+1}^i =
	\begin{cases}
		\dfrac{2}{3}|m_t^i|\left(1 + \cos\left(\dfrac{2\pi}{3} - \dfrac{2\varphi_{\delta}(m_t^i)}{3}\right)\right) & \text{if } m_t^i > p(\tilde{\delta}), \\
		0 & \text{if } |m_t^i| \leq p(\tilde{\delta}), \\
		-\dfrac{2}{3}|m_t^i|\left(1 + \cos\left(\dfrac{2\pi}{3} - \dfrac{2\varphi_{\delta}(m_t^i)}{3}\right)\right) & \text{if } m_t^i < -p(\tilde{\delta}),
	\end{cases}
\end{equation}
where $\varphi_{\tilde{\delta}}(m_t^i) = \arccos\left(\tfrac{\tilde{\delta}}{8}\left(\tfrac{|m_t^i|}{3}\right)^{-\frac{3}{2}}\right)$ and $p(\tilde{\delta}) = \tfrac{\sqrt[3]{54}}{4} \tilde{\delta}^{\frac{2}{3}}$.

For the $\bm{u}$-subproblem, let $c_{t+1} = \|\bm{x}_{t+1}\|_{1/2}^{1/2}$ and $\bm{d}_{t+1} = \bm{x}_{t+1} + \frac{\bm{\vartheta}_t}{\gamma}$. Then subproblem can be simplified to
\begin{equation}\label{eq:u-subproblem}
	\bm{u}_{t+1} = \arg\min_{\bm{u}\in \mathbb{R}^n} \; \zeta \frac{c_{t+1}}{\| \bm{u} \|_2^{1/2}} + \frac{\tilde{\delta}}{2} \| \bm{u} - \bm{d}_{t+1} \|_2^2.
\end{equation}
If $c_{t+1} = 0$, then $\bm{u}_{t+1} = \bm{d}_{t+1}$. If $\bm{d}_{t+1} = 0$, any $\bm{u}_{t+1}$ satisfying $\|\bm{u}_{t+1}\|_2^{1/2} = \left(\zeta c_{t+1} / (2\tilde{\delta})\right)^{1/5}$ is an optimal solution. For $\bm{d}_{t+1} \neq 0$ and $c_{t+1} \neq 0$, differentiating the objective function in \cref{eq:u-subproblem} yields $\left( -\frac{\zeta c_{t+1}}{2} \|\bm{u}\|_2^{-5/2} + \tilde{\delta} \right) \bm{u} = \tilde{\delta} \bm{d}_{t+1}$, which implies that $\bm{u} = \tau_{t+1} \bm{d}_{t+1}$ for some $\tau_{t+1} \geq 0$. Let $\eta_{t+1} = \|\bm{d}_{t+1}\|_2$. For $\eta_{t+1} > 0$, determining $\bm{u}$ reduces to finding $\tau_{t+1}$ as the root of $f_{\tau}(\tau_{t+1}) = \tau_{t+1}^5 - \tau_{t+1}^3 - \kappa _{t+1} = 0,$ where $\kappa _{t+1} = \frac{\zeta c_{t+1}}{2\tilde{\delta} (\eta_{t+1})^{5/2}} > 0$. The function $f_{\tau}(\tau_{t+1})$ is strictly unimodal on $(0, \infty)$, with a unique global minimum at $\tau^* = \sqrt{3/5}$. The function is strictly decreasing on $(0, \tau^*)$ and strictly increasing on $(\tau^*, \infty)$. For $\kappa _{t+1} > 0$, the equation $\tau_{t+1}^5 - \tau_{t+1}^3 = \kappa _{t+1}$ has a unique real root in $(1, \infty)$. Since general quintic equations lack closed-form solutions, numerical methods such as bisection or Newton method are used to approximate $\tau_{t+1} \in (1, \infty)$. In summary, the update formula for $\bm{u}$ is given by
\begin{equation}\label{eq:updata_subu}
	\bm{u}_{t+1} =
	\begin{cases}
		\bm{e}_{t+1} & \bm{d}_{t+1} = 0, \\
		\bm{d}_{t+1} & \bm{d}_{t+1} \neq 0, c_{t+1}=0,\\
		\tau_{t+1} \bm{d}_{t+1} & \bm{d}_{t+1} \neq 0, c_{t+1}\neq 0,
	\end{cases}
\end{equation}
where $\bm{e}_{t+1}$ is any vector satisfying $\|\bm{e}_{t+1}\|_2^{1/2} = \left(\zeta c_{t+1} / (2\tilde{\delta})\right)^{1/5}$. The overall ADMM framework for minimizing the unconstrained optimization problem \cref{eq:OptimizatopnObjectUnCon} is summarized in \cref{alg:admm_full}.
\begin{algorithm}[!h]
	\caption{ADMM with Inner Iteration for $\ell_{1/2}/\ell_2$ Minimization}
	\label{alg:admm_full}
	\begin{algorithmic}[1]
		\REQUIRE $\bm{A}, \bm{b}$; parameters $\zeta, \rho, \gamma, > 0$; $K$ (outer), $T$ (inner); tolerances $\epsilon_{\text{out}}, \epsilon_{\text{inner}}.$
		\STATE Initialize $\bm{x}_0, \bm{y}_0, \bm{\lambda}_0$
		\FOR{$k = 0$ \TO $K-1$}
		\STATE $\bm{\theta}_k \gets \bm{y}_k - \bm{\lambda}_k / \rho$
		\STATE Initialize inner variables $\bm{u}_0, \bm{\vartheta}_0$
		\FOR{$t = 0$ \TO $T-1$}
		\STATE Update $\bm{x}_{t+1}, \bm{u}_{t+1}, \bm{\vartheta}_{t+1}$ via \cref{eq:updata_subx},\cref{eq:updata_subu}, \cref{eq:inneradmm_v}
		\IF{$\|\bm{x}_{t+1} - \bm{u}_{t+1}\|_2 < \epsilon_{\text{inner}}$}
		\STATE \textbf{break inner loop}
		\ENDIF
		\ENDFOR
		\STATE $\bm{x}_{k+1} \gets \bm{x}_T$  \COMMENT{Use final inner iterate}
		\STATE Update $\bm{y}_{k+1}, \bm{\lambda}_{k+1}$ via \cref{eq:y-updata}, \cref{eq:admm_z}
		\IF{$\|\bm{x}_{k+1} - \bm{y}_{k+1}\|_2 < \epsilon_{\text{out}}$}
		\STATE \textbf{break outer loop}
		\ENDIF
		\ENDFOR
		\RETURN $\bm{x}_K, \bm{y}_K$
	\end{algorithmic}
\end{algorithm}

\section{Convergence analysis}\label{sec:convergence}

This section analyzes the convergence and convergence rate of the ADMM framework introduced in \cref{eq:admm_updates}. The analysis builds upon the framework and techniques established in \cite{lou2018fast,pang2018decomposition,tao2022minimization,wang2019global, zeng2021analysis}. 

We initially present the formal definitions of the Kurdyka-\L{}ojasiewicz (K\L{}) property and the K\L{} exponent \cite{bolte2007lojasiewicz,kurdyka1998gradients,lojasiewicz1963propriete}, which will be employed to establish results concerning convergence and the convergence rate.

\begin{definition}[K\L{} property]\label{def:kl}
	Let $\varphi:\mathbb{R}^n \to (-\infty,+\infty]$ be proper and l.s.c. We say that $\varphi$ satisfies the \emph{K\L{} property} at $\hat{\bm{x}} \in \mathrm{dom}\,\partial \varphi$ if $\exists \, \tilde{c} \in (0,\infty],\, r > 0$, and a concave $\psi:[0,\tilde{c}) \to [0,\infty)$ with $\psi(0)=0$, $\psi \in C^1((0,\tilde{c}))$, $\psi'>0$, such that $\psi'\left(\varphi(\bm{x})-\varphi(\hat{\bm{x}})\right)\cdot \mathrm{dist}(0,\partial \varphi(\bm{x})) \,\geq\, 1$
	for all $\bm{x} \in \mathbb{B}_r(\hat{\bm{x}})$ with $\varphi(\hat{\bm{x}})<\varphi(\bm{x})<\varphi(\hat{\bm{x}})+\tilde{c}$.  
\end{definition}

If the inequality holds for all $\bm{x} \in \Omega \cap \mathbb{B}_r(\hat{\bm{x}})$ with $\Omega \subseteq \mathbb{R}^n$ closed, then $\varphi$ satisfies the K\L{} property on $\Omega$ at $\hat{\bm{x}}$.  
\begin{definition}[K\L{} exponent]\label{def:kl-exponent}
	Let $\varphi:\mathbb{R}^n \to (-\infty,+\infty]$.  
	We say that $\varphi$ satisfies the K\L{} property at $\hat{\bm{x}}\in\mathrm{dom}\,\partial \varphi$ with exponent $s$ if there exist constants $\tilde{\omega} >0$, $\epsilon>0$, and $\tilde{c}>0$ such that $\mathrm{dist}(0,\partial \varphi(\bm{x})) \,\geq\, \tilde{\omega}(\varphi(\bm{x})-\varphi(\hat{\bm{x}}))^{\upsilon},$ for all $\bm{x}\in\mathbb{B}_\epsilon(\hat{\bm{x}})$ with   $\varphi(\hat{\bm{x}})<\varphi(\bm{x})<\varphi(\hat{\bm{x}})+\tilde{c}$.
\end{definition}

If a common $\upsilon \in [0,1)$ holds for all $\hat{\bm{x}} \in \mathrm{dom}\,\partial \varphi$, then $\varphi$ is a K\L{} function with exponent $\upsilon$. For any closed $\Omega \subseteq \mathbb{R}^n$, the property holds on $\Omega$ at $\hat{\bm{x}}$ if the inequality is satisfied for all $\bm{x} \in \Omega \cap \mathbb{B}_r(\hat{\bm{x}})$.
\subsection{Existence of Solutions}
This subsection focuses on the existence of solutions to the optimization problem \cref{eq:OptimizatopnObjectUnCon}.

\begin{theorem}\label{thm:SolutionExistence}
	Let $ \bm{A} \in \mathbb{R}^{m \times n} $, $ \bm{b} \in \mathbb{R}^m $, and $ \zeta > 0 $. Consider the nonconvex optimization problem \cref{eq:OptimizatopnObjectUnCon}. If $ \bm{b} \notin ker(\bm{A}^\top) $, there exists a direction $ \bm{\tilde{v}} \in \mathbb{R}^n \setminus \{\bm{0}\} $ and a constant $ \epsilon^* > 0 $, such that for all $ \epsilon \in (0, \epsilon^*) $, we have $ \mathcal{H}(\epsilon \bm{\tilde{v}}) < \mathcal{H}(\bm{0})$.
\end{theorem}

\begin{proof}
	The proof is structured into two principal steps.
	
	\paragraph{(i) Existence of direction $ \bm{\tilde{v}} $}
	Since $ \bm{b} \notin ker(\bm{A}^\top) $, we have $ \bm{A}^\top \bm{b} \neq 0 $, which implies there exists $ \bm{\tilde{v}} \in \mathbb{R}^n $ such that $\langle \bm{A\tilde{v}}, \bm{b} \rangle = \bm{\tilde{v}}^\top \bm{A}^\top \bm{b} \neq 0$. To ensure that the perturbation decreases the objective function, we may assume without loss of generality that $ \langle \bm{A\tilde{v}}, \bm{b} \rangle > 0 $, otherwise we replace $ \bm{\tilde{v}} $ by $ -\bm{\tilde{v}} $. Hence, there exists a direction $ \bm{\tilde{v}} \in \mathbb{R}^n \setminus \{\bm{0}\} $ such that $ \langle \bm{A\tilde{v}}, \bm{b} \rangle > 0 $.
	
	\paragraph{(ii) Perturbation construction and descent analysis}
	Define the perturbed point $ \bm{x}_\epsilon := \epsilon \bm{\tilde{v}} $, where $ \epsilon > 0 $. Then the objective function becomes $ \mathcal{H}(\bm{x}_\epsilon) = \zeta  \frac{\|\bm{x}_\epsilon\|_{1/2}^{1/2}}{\|\bm{x}_\epsilon\|_2^{1/2}} + \frac{1}{2} \|\bm{Ax}_\epsilon - \bm{b}\|_2^2. $ Let $ C_{\bm{\tilde{v}}} := \frac{\|\bm{\tilde{v}}\|_{1/2}^{1/2}}{\|\bm{\tilde{v}}\|_2^{1/2}} \geq 1, \beta_{\bm{\tilde{v}}} := \langle \bm{A\tilde{v}}, \bm{b} \rangle > 0, \alpha_{\bm{\tilde{v}}} := \|\bm{A\tilde{v}}\|_2^2 > 0$ and $\delta_{\bm{\tilde{v}}} := \zeta(C_{\bm{\tilde{v}}} - 1) \geq 0$. Then we have $
	\frac{\|\bm{x}_\epsilon\|_{1/2}}{\|\bm{x}_\epsilon\|_2^{1/2}} = C_{\bm{\tilde{v}}}$ and $\|\bm{Ax}_\epsilon - \bm{b}\|_2^2 = \epsilon^2 \alpha_{\bm{\tilde{v}}} - 2\epsilon \beta_{\bm{\tilde{v}}} + \|\bm{b}\|_2^2 $.
	Thus, the change is $ \mathcal{H}(\bm{x}_\epsilon) - \mathcal{H}(\bm{0}) = \zeta(C_{\bm{\tilde{v}}} - 1) - \epsilon \beta_{\bm{\tilde{v}}} + \frac{1}{2} \epsilon^2 \alpha_{\bm{\tilde{v}}} =: \phi(\epsilon)$. It follows that $ \phi(\epsilon) = \delta_{\bm{\tilde{v}}} - \beta_{\bm{\tilde{v}}} \epsilon + \frac{1}{2} \alpha_{\bm{\tilde{v}}} \epsilon^2 $ is a quadratic function. The discriminant of $ \phi(\epsilon) $ is given by $ \Delta := \beta_{\bm{\tilde{v}}}^2 - 2 \alpha_{\bm{\tilde{v}}} \delta_{\bm{\tilde{v}}}.$
	If $ \Delta > 0 $, then $ \phi(\epsilon) $ has two real roots, and the smaller positive root is $ \epsilon^* := \frac{\beta_{\bm{\tilde{v}}} - \sqrt{\beta_{\bm{\tilde{v}}}^2 - 2 \alpha_{\bm{\tilde{v}}} \delta_{\bm{\tilde{v}}}}}{\alpha_{\bm{\tilde{v}}}} > 0. $
	Therefore, for all $ \epsilon \in (0, \epsilon^*) $, we have $ \phi(\epsilon) < 0 $, i.e., $ \mathcal{H}(\epsilon \bm{\tilde{v}}) < \mathcal{H}(\bm{0})$, which implies that $\bm{0}$ is not a global minimizer.
\end{proof}

\begin{theorem}\label{thm:solutionnonempty}
	Let $ \bm{A} \in \mathbb{R}^{m \times n} $ be a matrix of full row rank ($\text{rank}(\bm{A}) = m < n $), $\mathbf{b} \in \mathbb{R}^m \setminus \{ \bm{0} \}$, and the parameter $\zeta \in \mathbb{R}$ satisfy $ 0 < \zeta < \frac{\|\bm{b}\|_2^2}{2(n^{3/4} - 1)} $. For the optimization problem \cref{eq:OptimizatopnObjectUnCon}, the optimal solution set $\arg \min_{\bm{x} \in \mathbb{R}^n} \mathcal{H}(\bm{x})$ is non-empty.
\end{theorem}

\begin{proof}
	By \cref{lem:RatioBounds}, we know that $\|\bm{x}\|_{1/2}^{1/2}/\|\bm{x}\|_2^{1/2} \in [1,n^{3/4}]$ for any $\bm{x} \neq \bm{0}$. Moreover, the objective function $\mathcal{H}$ is proper and bounded from below, and hence the infimum $\mathcal{H}^*=\inf_{\bm{x} \in \mathbb{R}^n} \mathcal{H}(\bm{x})$ is finite and $ \mathcal{H}^* \geq \zeta > 0 $. 
	
	Let $\{\bm{x}_k\}$ be a minimizing sequence with $\mathcal{H}(\bm{x}_k)\to \mathcal{H}^*$. If $\{\bm{x}_k\}$ is bounded, then compactness in finite-dimensional spaces ensures the existence of a convergent subsequence, and lower semicontinuity of $\mathcal{H}$ implies that the limit point is a minimizer. Thus, it suffices to exclude the possibility that $\{\bm{x}_k\}$ is unbounded. Suppose, to the contrary, that $\|\bm{x}_k\|_2 \to \infty$. Define the normalized sequence $\bm{\tilde{u}}_k = \bm{x}_k/\|\bm{x}_k\|_2$. Since $\|\bm{\tilde{u}}_k\|_2=1$, compactness of the unit sphere yields a subsequence converging to some $\bm{\tilde{u}}$ with $\|\bm{\tilde{u}}\|_2=1$. If $\bm{\tilde{u}} \notin \text{ker}(\bm{A})$, then $\|\bm{A\tilde{u}}\|_2>0$. By continuity, we have $\|\bm{A \tilde{u}}_k\|_2 \to \|\bm{A \tilde{u}}\|_2$. Therefore, there exist a constant $c_1 > 0$ and an index $K$ such that for all $k \geq K$, it holds that $|\bm{A \tilde{u}}_k|_2 \geq c_1$. Consequently, for these $k$, we obtain
	\begin{equation*}
		\|\bm{Ax}_k - b\|_2 \geq \|\bm{x}_k\|_2 \|\bm{A\tilde{u}}_k\|_2 - \|\bm{b}\|_2 \geq c_1 \|\bm{x}_k\|_2 - \|\bm{b}\|_2\to \infty
	\end{equation*}
	as $\|\bm{x}_k\|_2 \to \infty$, contradicting the boundedness of $\mathcal{H}^*$ since $\mathcal{H}(\bm{x}_k) \to \mathcal{H}^*<\infty$. This implies that $\bm{\tilde{u}} \in \text{ker}(\bm{A})$, then $\|\bm{A\tilde{u}}\|_2=0$. It follows that $\tfrac{1}{2}\|\bm{A}\bm{x}_k-\bm{b}\|_2^2 \to \tfrac{1}{2}\|\bm{b}\|_2^2$, while the ratio $\|\bm{x}_k\|_{1/2}^{1/2}/\|\bm{x}_k\|_2^{1/2}$ converges to $\|\bm{\tilde{u}}\|_{1/2}^{1/2}$. Moreover, for any unit vector $\bm{\tilde{u}}$, it holds that $\|\bm{\tilde{u}}\|_{1/2}\ge\|\bm{\tilde{u}}\|_1 \ge \|\bm{\tilde{u}}\|_2 = 1$. Thus, we obtain $\liminf_{k\to\infty}\mathcal{H}(\bm{x}_k) \ge \zeta + \tfrac{1}{2}\|\bm{b}\|_2^2$. In addition, since $\bm{A}$ has full row rank, the system $\bm{Ax}=\bm{b}$ admits a solution $\check{\bm{x}} = \bm{A}^\top(\bm{AA}^\top)^{-1}\bm{b}$, for which $\mathcal{H}(\check{\bm{x}}) = \zeta \frac{\|\check{\bm{x}}\|_{1/2}^{1/2}}{\|\check{\bm{x}}\|_2^{1/2}} \leq \zeta n^{3/4}$, and hence $ \mathcal{H}^* \leq \zeta n^{3/4}$. By combining the two bounds, we have $\zeta n^{3/4} \ge \mathcal{H}^* \ge \zeta + \tfrac12\|\bm{b}\|_2^2$, which implies $\zeta(n^{3/4} - 1)\ge \tfrac{1}{2}\|\bm{b}\|_2^2$, contradicting the assumed condition on $\zeta$. Thus, $\{\bm{x}_k\}$ must be bounded, and the existence of a convergent subsequence yields a global minimizer of \cref{eq:OptimizatopnObjectUnCon}, which means the optimal solution set $\arg \min_{\bm{x} \in \mathbb{R}^n} \mathcal{H}(\bm{x})$ is non-empty.
\end{proof}

\subsection{Subsequential Convergence}
In this subsection, we discuss the convergence of subsequences. To this end, we first show two lemmas.
\begin{lemma}[Sufficient Descent]\label{lem:SufficientDescent}
	Let $\{(\bm{x}_k, \bm{y}_k, \bm{\lambda}_k)\}$ denote the sequence produced by the iterative scheme \cref{eq:admm_updates} and $\mathcal{L}_\rho(\bm{x}_k, \bm{y}_k, \bm{\lambda}_k)=\mathcal{L}_\rho^{k}$. For the augmented Lagrangian function $\mathcal{L}_\rho^k$, it then follows that $\mathcal{L}_\rho^{k+1} \geq \mathcal{L}_\rho^k +  \left( \frac{\lambda_{\min}(\bm{A}^\top \bm{A}) + \rho}{2} - \frac{L_g^2}{\rho} \right) \|\bm{y}_{k+1} - \bm{y}_k\|^2,$
	where $L_g = \lambda_{\max}(\bm{A}^\top \bm{A})$ denotes the largest eigenvalue of $\bm{A}^\top \bm{A}$. If $ \rho > \frac{-\lambda_{\min}(\bm{A}^\top \bm{A})+\sqrt{\left(\lambda_{\min}(\bm{A}^\top \bm{A})\right)^2 + 8L_g^2}}{2}$, then the sequence $ \{\mathcal{L}_\rho\}$ is nonincreasing.
\end{lemma}
\begin{proof}
	This proof is similar to that in \cite[Lemma 2.1]{wang2023variant}, and hence omitted here.
\end{proof}

\begin{lemma}\label{lem:LsubgradientBounded}
	Let $\{\left(\bm{x}_k, \bm{y}_k, \bm{\lambda}_k\right)\}$ be the sequence generated by \cref{alg:admm_full}, then there exists a constant $C_{\rho} > 0$ such that $\mathrm{dist}^2(\mathbf{0}, \partial \mathcal{L}_{\rho}^{k+1}) \leq C_{\rho} \|\bm{y}_{k+1} - \bm{y}_k\|_2^2.$
\end{lemma}

\begin{proof}
	According to the augmented Lagrangian function \cref{eq:AugmentedLagrangian}, we have
	\begin{equation}\label{eq:DervAugmentedLagrangian}
		\begin{aligned}
			\frac{\partial \mathcal{L}_{\rho}^{k+1}}{\partial \bm{x}_{k+1}} 
			&= \zeta \partial \left( \frac{\|\bm{x}_{k+1}\|_{1/2}^{1/2}}{\|\bm{x}_{k+1}\|_2^{1/2}} \right) + \bm{\lambda}_{k+1} + \rho (\bm{x}_{k+1} - \bm{y}_{k+1}), \\
			\frac{\partial \mathcal{L}_{\rho}^{k+1}}{\partial \bm{y}_{k+1}} 
			&= \nabla g(\bm{y}_{k+1}) - \bm{\lambda}_{k+1} - \rho (\bm{x}_{k+1} - \bm{y}_{k+1}),\;\;
			\frac{\partial \mathcal{L}_{\rho}^{k+1}}{\partial \bm{\lambda}_{k+1}} 
			= \bm{x}_{k+1} - \bm{y}_{k+1},
		\end{aligned}
	\end{equation} 
	and by the optimality condition of \cref{eq:admm_updates}, we know that
	\begin{equation}
		\bm{0} \in \zeta \partial \left( \frac{\|\bm{x}_{k+1}\|_{1/2}^{1/2}}{\|\bm{x}_{k+1}\|_2^{1/2}} \right) + \rho \left( \bm{x}_{k+1} - \bm{y}_k + \frac{\bm{\lambda}_k}{\rho} \right)
	\end{equation}
	and $\nabla g(\bm{x}_{k+1}) - \rho \left( \bm{x}_{k+1} - \bm{y}_{k+1} + \frac{\bm{\lambda}_k}{\rho} \right) = \bm{0}.$ Let $\bm{\xi}^{k+1} = (\bm{\xi}_1^{k+1}, \bm{\xi}_2^{k+1}, \bm{\xi}_3^{k+1})^\top$. Here, $\bm{\xi}_1^{k+1} = \rho (\bm{y}_k - \bm{y}_{k+1}) + (\bm{\lambda}_{k+1} - \bm{\lambda}_k), \bm{\xi}_2^{k+1} = \bm{\lambda}^k - \bm{\lambda}^{k+1}$ and $\bm{\xi}_3^{k+1} = \bm{x}_{k+1} - \bm{y}_{k+1} = (\bm{\lambda}_{k+1} - \bm{\lambda}_k)/\rho.$ It is easy to get that $\bm{\xi}^{k+1} \in \partial \mathcal{L}_{\rho}{\left(\bm{x}_{k+1}, \bm{y}_{k+1}, \bm{\lambda}_{k+1}\right)}$. Moreover, $\|\bm{\xi}^{k+1}\|_2^2 = \|\bm{\xi}_1^{k+1}\|_2^2 + \|\bm{\xi}_2^{k+1}\|_2^2 + \|\bm{\xi}_3^{k+1}\|_2^2$, one has
	\begin{equation*}
		\begin{aligned}
			\|\rho (\bm{y}_k - \bm{y}_{k+1}) &+ (\bm{\lambda}_{k+1} - \bm{\lambda}_k)\|_2^2 + \left(1 + \frac{1}{\rho^2}\right) \|\bm{\lambda}_{k+1} - \bm{\lambda}_k\|_2^2 \\
			&\overset{a}{\leq} \left(\left(\rho+L_g\right)\|\bm{y}_{k+1} - \bm{y}_k\|_2\right)^2 + \left(1 + \frac{1}{\rho^2}\right) \|\bm{\lambda}_{k+1} - \bm{\lambda}_k\|_2^2\\ 
			&\overset{b}{\leq} \left( \left(\rho+L_g\right)^2 + L_g^2 \left(1 + \frac{1}{\rho^2}\right) \right) \|\bm{y}_{k+1} - \bm{y}_k\|_2^2 = C_{\rho} \|\bm{y}_{k+1} - \bm{y}_k\|_2^2,
		\end{aligned}
	\end{equation*}
	where $C_{\rho} := \left(\rho+L_g\right)^2 + L_g^2 \left(1 + \frac{1}{\rho^2}\right)$, $\overset{a}{\leq}$ is obtained by the triangle inequality, and $\overset{b}{\leq}$ is obtained by $\|\bm{\lambda}_{k+1} - \bm{\lambda}_{k}\|_2 = \|\nabla g(\bm{y}_{k+1}) - \nabla g(\bm{y}_k)\|_2 \leq L_g \|\bm{y}_{k+1} - \bm{y}_k\|_2.$
\end{proof}

\begin{theorem}\label{thm:admm_convergence}
	Let $\{(\bm{x}_k, \bm{y}_k, \bm{\lambda}_k)\}$ denote the sequence generated by the ADMM algorithm with iterative updates as defined in \cref{eq:admm_updates}. Suppose that $\{\bm{x}_k\}$ is bounded and the penalty parameter $\rho$ satisfies $\rho > \frac{-\lambda_{\min}\left(\bm{A}^{\top}\bm{A}\right) + \sqrt{\left(\lambda_{\min}\left(\bm{A}^{\top}\bm{A}\right)\right)^2 + 8L_g^2}}{2}.$ Then, the following statements hold:
	\begin{enumerate}
		\item[\textup{(i)}] \textbf{Boundedness:} The primal variable sequence $\{\bm{y}_k\}$ and the dual multiplier sequence $\{\bm{\lambda}_k\}$ are both bounded in $\mathbb{R}^n$.
		\item[\textup{(ii)}] \textbf{Existence:} The composite sequence $\{(\bm{x}_k, \bm{y}_k, \bm{\lambda}_k)\}$ admits at least one accumulation point in $\mathbb{R}^n \times \mathbb{R}^n \times \mathbb{R}^n$.
		\item[\textup{(iii)}] \textbf{Asymptotic:} The successive differences converge to zero asymptotically: $ \lim\limits_{k \to \infty} \|\bm{x}_k - \bm{x}_{k+1}\|_2 = 0, \lim\limits_{k \to \infty} \|\bm{y}_k - \bm{y}_{k+1}\|_2 = 0,$  and $\lim\limits_{k \to \infty} \|\bm{\lambda}_k - \bm{\lambda}_{k+1}\|_2 = 0.$
		\item[\textup{(iv)}] Any accumulation point $(\bm{x}^*, \bm{y}^*, \bm{\lambda}^*)$ of the composite sequence $\{(\bm{x}_k, \bm{y}_k, \bm{\lambda}_k)\}$ satisfies the first-order necessary conditions for stationarity of the augmented Lagrangian function $\mathcal{L}_\rho$, and $\bm{x}^*$ is a stationary point of the problem \cref{eq:OptimizatopnObjectUnCon}.
	\end{enumerate}
\end{theorem}

\begin{proof}
	(i) We first show that the sequences $\{\bm{y}_k\}$ and $\{\bm{\lambda}_k\}$ are bounded. By \cref{lem:SufficientDescent}, if $ \rho > \frac{-\lambda_{\min}\left(\bm{A}^{\top}\bm{A}\right)+\sqrt{\left(\lambda_{\min}\left(\bm{A}^{\top}\bm{A}\right)\right)^2 + 8L_g^2}}{2}$, then $\{\mathcal{L}_{\rho}(\bm{x}_k, \bm{y}_k, \bm{\lambda}_k)\}$ is monotonically decreasing . So we have $\mathcal{L}_{\rho}(\bm{x}_k, \bm{y}_k, \bm{\lambda}_k) \leq \mathcal{L}_{\rho}(\bm{x}_0, \bm{y}_0, \bm{\lambda}_0),$ which shows that the sequence $\{\mathcal{L}_{\rho}(\bm{x}_k, \bm{y}_k, \bm{\lambda}_k)\}$ is upper bounded. On the other hand, by setting $ \bm{y}_{k+1}=\bm{x}_{k}$, we have 
	\begin{equation}\label{eq:lowerbounded}
		\begin{aligned}
			g(\bm{y}_k) - g(\bm{x}_k) & - \langle \nabla g(\bm{y}_k), \bm{y}_k - \bm{x}_k \rangle \geq \frac{\lambda_{\min}\left(\bm{A}^{\top}\bm{A}\right)}{2}\|\bm{y}_k - \bm{x}_{k} \|_2^2.
		\end{aligned}
	\end{equation}
	Combining \cref{eq:admm_updates} and \cref{eq:lowerbounded}, one has
	\begin{equation*}
		\begin{aligned}
			\mathcal{L}_{\rho}\left(\bm{x}_k, \bm{y}_k, \bm{\lambda}_k \right) \geq \zeta \frac{\|\bm{x}_k\|_{1/2}^{1/2}}{\|\bm{x}_k\|_{2}^{1/2}} + g(\bm{x}_k) + \frac{\lambda_{\min}\left(\bm{A}^{\top}\bm{A}\right)+\rho}{2}\|\bm{y}_k - \bm{x}_{k} \|_2^2 .
		\end{aligned}
	\end{equation*}
	It can be deduced that, whenever $\rho > \frac{\lambda_{\min}\left(\bm{A}^{\top}\bm{A}\right) + \sqrt{\left(\lambda_{\min}\left(\bm{A}^{\top}\bm{A}\right)\right)^2 + 8L_g^2}}{2},$ the augmented Lagrangian function satisfies $\mathcal{L}_{\rho}(\bm{x}_k, \bm{y}_k, \bm{\lambda}_k) \geq 0,$ which implies that the sequence $\{\mathcal{L}_{\rho}(\bm{x}_k, \bm{y}_k, \bm{\lambda}_k)\}$ remains bounded. Moreover, if the iterate sequence $\{\bm{x}_k\}$ is bounded, then the term
	$\|\bm{x}^k - \bm{y}_k\|_2^2$ is bounded, which directly ensures that $\{\bm{y}_k\}$ is bounded as well. Finally, invoking the multiplier update rule $\bm{\lambda}_k = \nabla g(\bm{y}_k) = \bm{A}^{\top}(\bm{A}\bm{y}_k - \bm{b})$, it follows that $\|\bm{\lambda}_k\|_2^2 = \|\nabla \Phi(\bm{y}_k) - \nabla g(\bm{0}) + \nabla g(\bm{0})\|_2^2 \leq 2\left(L_g^2 \|\bm{y}_k\|_2^2 + \lambda_{\max}(\bm{A}^\top \bm{A}) \|\bm{b}\|_2^2\right),$	which shows that $\{\bm{\lambda}_k\}$ is bounded.
	
	(ii) It follows from (i) that the composite sequence $\{\left(\bm{x}_k, \bm{y}_k, \bm{\lambda}_k\right)\}$ is bounded and hence it has at least one accumulation point.
	
	(iii) Summing the sufficient descent inequality in \cref{lem:SufficientDescent} over $ j = 0 $ to $ k $, we obtain $\mathcal{L}_{\rho}^{k+1} \leq \mathcal{L}_{\rho}^{0} - \sum_{j=0}^k \left( \frac{\lambda_{\min}(\bm{A}^\top \bm{A}) + \rho}{2} - \frac{L_g^2}{\rho} \right) \|\bm{y}_{j+1} - \bm{y}_j\|_2^2.$
	Taking the limit as $ k \to \infty $, and using the nonnegativity of the augmented Lagrangian function $ \mathcal{L}_{\rho} $ established earlier, the left-hand side remains bounded from below. Consequently, the series $\sum_{j=0}^\infty \left( \frac{\lambda_{\min}(\bm{A}^\top \bm{A}) + \rho}{2} - \frac{L_g^2}{\rho} \right) \|\bm{y}_{j+1} - \bm{y}_j\|_2^2$ converges, which, under the given condition on $ \rho $, implies that $ \|\bm{y}_{j+1} - \bm{y}_j\|_2^2 \to 0 $ as $ j \to \infty $. It follows from the Lipschitz continuity of $ \nabla g $ that $ \|\bm{\lambda}_{k+1} - \bm{\lambda}_k\|_2 \to 0 $. Furthermore, from the update rule \cref{eq:admm_z}, we have $\bm{x}_{k} - \bm{x}_{k+1} = \rho^{-1}(\bm{\lambda}_{k+1} - \bm{\lambda}_k) - \rho^{-1}(\bm{\lambda}_k - \bm{\lambda}_{k-1}) + (\bm{y}_{k+1} - \bm{y}_k),$ and thus, as each term on the right-hand side tends to zero in norm, we conclude that $ \|\bm{x}_{k+1} - \bm{x}_k\|_2^2 \to 0 $.
	
	(iv) From part (i), the composite sequence $ \{(\bm{x}_k, \bm{y}_k, \bm{\lambda}_k)\} $ is bounded. By the Bolzano--Weierstrass theorem, there exists a subsequence $ \{(\bm{x}_{k_i}, \bm{y}_{k_i}, \bm{\lambda}_{k_i})\} $ converging to some limit point $ (\bm{x}^*, \bm{y}^*, \bm{\lambda}^*) $ as $ i \to \infty $. As shown in (iii), the shifted subsequence $ (\bm{x}_{k_i+1}, \bm{y}_{k_i+1}, \bm{\lambda}_{k_i+1}) $ converges to the same limit. Invoking \cref{lem:LsubgradientBounded}, we deduce that $ \mathbf{0} \in \partial \mathcal{L}_{\rho}(\bm{x}^*, \bm{y}^*, \bm{\lambda}^*) $, confirming that $ (\bm{x}^*, \bm{y}^*, \bm{\lambda}^*) $ is a stationary point of the augmented Lagrangian function \cref{eq:innerADMM}. Finally, from the stationarity conditions derived in \cref{eq:DervAugmentedLagrangian}, it follows that $\bm{x}^* = \bm{y}^*$ and $\zeta \, \partial \left( \frac{\|\bm{x}^*\|_{1/2}^{1/2}}{\|\bm{x}^*\|_2^{1/2}} \right) + \nabla g(\bm{x}^*) = \bm{0}$, which characterizes $ \bm{x}^* $ as a stationary point of the problem \cref{eq:OptimizatopnObjectUnCon}.  
\end{proof}

\subsection{Global convergence and convergence rate}

Inspired by \cite{tao2022minimization}, we introduce a new merit function tailored for analyzing the convergence properties of the ADMM scheme \cref{eq:admm_updates}, defined as
\begin{equation}\label{eq:MeritFunction}
	\mathcal{L}_m(\bm{x}, \bm{y}) 
	= \zeta\frac{\|\bm{x}\|_{1/2}^{1/2}}{\|\bm{x}\|_{2}^{1/2}} + g(\bm{x}) 
	+ \frac{\rho}{2} \|\bm{x} - \bm{y}\|_2^2 .
\end{equation}
Next,we begin with several auxiliary lemmas, which, in conjunction with the K\L{} property and the associated K\L{} inequality, will be employed to establish both global convergence and a R-linear convergence.

\begin{lemma}\label{lem:MSufficientDescent}
	Consider the sequence $\{(\bm{x}_k, \bm{y}_k)\}$ generated by the ADMM updates in \cref{eq:admm_updates}, and let $\mathcal{L}_m(\bm{x}, \bm{y})$ be the merit function defined in \cref{eq:MeritFunction}. Suppose that $\lambda_{\min}(\bm{A}^\top \bm{A}) > \frac{L_g^3 + 2L_g^2 \rho - \rho^3}{\rho^2}$, then there exists a constant $\tilde{\kappa}_1 > 0$ such that $\mathcal{L}_m(\bm{x}, \bm{y})$ satisfies the sufficient descent property
	\begin{equation}\label{eq:MeritFunctionSufficentDec}
		\mathcal{L}_m(\bm{x}_{k+1}, \bm{y}_{k+1}) 
		\leq \mathcal{L}_m(\bm{x}_k, \bm{y}_k) 
		- \tilde{\kappa}_1 \|\bm{y}_{k+1} - \bm{y}_k\|_2^2,
	\end{equation}
	where $	\tilde{\kappa}_1 = \frac{\lambda_{\min}(\bm{A}^\top \bm{A}) + \rho}{2} - \frac{L_g^2}{\rho} - \frac{L_g^3}{2\rho^2}.$
\end{lemma}

\begin{proof}
    For brevity, denote $\mathcal{L}_{\rho}^k = \mathcal{L}_{\rho}(\bm{x}_k, \bm{y}_k, \bm{\lambda}_k)$. Using the update $\bm{\lambda}_{k+1} = \nabla g(\bm{y}_{k+1})$ and the quadratic upper bound induced by the $L_g$-smoothness of $g$, we have $g(\bm{x}_{k+1}) \leq g(\bm{y}_{k+1}) + \langle \nabla g(\bm{y}_{k+1}), \bm{x}_{k+1} - 	\bm{y}_{k+1} \rangle + \frac{L_g}{2} \|\bm{x}_{k+1} - \bm{y}_{k+1}\|_2^2.$
	It follows from $\mathcal{H}(\bm{x}) = \zeta \frac{\|\bm{x}\|_{1/2}^{1/2}}{\|\bm{x}\|_2^{1/2}} + g(\bm{x})$ that
	\begin{align*}
		\mathcal{L}_{\rho}^{k+1} 
		&= \zeta \frac{\|\bm{x}_{k+1}\|_{1/2}^{1/2}}{\|\bm{x}_{k+1}\|_2^{1/2}} + g(\bm{y}_{k+1}) + \langle \nabla g(\bm{y}_{k+1}), \bm{x}_{k+1} - \bm{y}_{k+1} \rangle + \frac{\rho}{2} \|\bm{x}_{k+1} - \bm{y}_{k+1}\|_2^2 \\
		&\geq \mathcal{H}(\bm{x}_{k+1}) + \frac{\rho - L_g}{2} \|\bm{x}_{k+1} - \bm{y}_{k+1}\|_2^2.
	\end{align*}
	
	On the other hand, by convexity of $g$ and $\bm{\lambda}_k = \nabla g(\bm{y}_k)$, we obtain $g(\bm{y}_k) + \langle \nabla g(\bm{y}_k), \bm{x}_k - \bm{y}_k \rangle \leq g(\bm{x}_k),$
	which implies $\mathcal{L}_{\rho}^k \leq \mathcal{H}(\bm{x}_k) + \frac{\rho}{2} \|\bm{x}_k - \bm{y}_k\|_2^2.$ For notational convenience, let $\mathcal{L}_m^{k+1} := \mathcal{L}_m(\bm{x}_{k+1}, \bm{y}_{k+1})$.
	 Combining with \cref{lem:SufficientDescent} and the update rule in \cref{eq:admm_z}, the above two inequalities imply that
	\begin{equation}\label{eq:descent_L}
		\mathcal{L}_{m}^{k+1} \leq \mathcal{L}_{m}^k - \left( \frac{\lambda_{\min}(\bm{A}^\top \bm{A}) + \rho}{2} - \frac{L_g^2}{\rho} \right) \|\bm{y}_{k+1} - \bm{y}_k\|_2^2 + \frac{L_g}{2\rho^2} \|\bm{\lambda}_{k+1} - \bm{\lambda}_k\|_2^2.
	\end{equation}
	It follows from $\|\nabla g(\bm{y}_{k+1}) - \nabla g(\bm{y}_k)\| \leq L_g \|\bm{y}_{k+1} - \bm{y}_k\|$ and $\bm{\lambda}_k = \nabla g(\bm{y}_k)$ that 
	\begin{equation}\label{eq:MsfficentDecent}
		\mathcal{L}_m^{k+1} \leq \mathcal{L}_m^k - \left( \frac{\lambda_{\min}\left(\bm{A}^{\top}\bm{A}\right) + \rho}{2} - \frac{2\rho L_g^2 + L_g^3}{2\rho^2} \right) \|\bm{y}_k - \bm{y}_{k+1}\|_2^2.
	\end{equation}
	Let $ \tilde{\kappa}_1 = \frac{\lambda_{\min}\left(\bm{A}\bm{A}^{\top}\right) + \rho}{2} - \frac{2\rho L_g^2 + L_g^3}{2\rho^2}$, it further follows from \cref{eq:MsfficentDecent} that the inequality \cref{eq:MeritFunctionSufficentDec} holds with $\tilde{\kappa}_1$, which is positive when $\lambda_{\min}(\bm{A}^\top \bm{A}) > \frac{L_g^3 + 2L_g^2 \rho - \rho^3}{\rho^2}$. This establishes the sufficient descent property \cref{eq:MeritFunctionSufficentDec} and completes the proof.
\end{proof}

\begin{lemma}\label{lem:subgradient_bound}
	Let $\{(\bm{x}_k, \bm{y}_k)\}$ be the sequence generated by the update scheme in \cref{eq:admm_updates}, then there is a positive constant $\tilde{\kappa}_2$ satisfying
	\begin{equation}
		\operatorname{dist}^2\left(\mathbf{0}, \partial \mathcal{L}_m(\bm{x}_{k+1}, \bm{y}_{k+1})\right) \leq \tilde{\kappa}_2 \|\bm{y}_{k+1} - \bm{y}_k\|_2^2
	\end{equation}
	for all $k \geq 0$, where $\mathcal{L}_m$ denotes the associated merit function.
\end{lemma}

\begin{proof}
	Consider the optimality condition of \cref{eq:admm_x} together with the multiplier update $\bm{\lambda}_{k+1} - \bm{\lambda}_k = \rho (\bm{x}_{k+1} - \bm{y}_{k+1})$, which yields $\bm{0} \in \zeta \, \partial \left(\frac{\|\bm{x}_{k+1}\|_{1/2}^{1/2}}{\|\bm{x}_{k+1}\|_2^{1/2}}\right) + \bm{\lambda}_{k+1} + \rho (\bm{y}_{k+1} - \bm{y}_k).$ Let $\bm{d}_1^{k+1} := \bm{\lambda}_{k+1} - \bm{\lambda}_k + \rho (\bm{y}_{k+1} - \bm{y}_k)$.	Then
	\begin{equation*}
		\bm{d}_1^{k+1} \in \zeta \, \partial \left(\frac{\|\bm{x}_{k+1}\|_{1/2}^{1/2}}{\|\bm{x}_{k+1}\|_2^{1/2}}\right) + \nabla g(\bm{x}_{k+1}) + \rho (\bm{x}_{k+1} - \bm{y}_{k+1}),
	\end{equation*}
	which shows that $\bm{d}_1^{k+1} \in \partial_{\bm{x}} \mathcal{L}_m^{k+1}$. Similarly, let $\bm{d}_2^{k+1} := \rho (\bm{y}_{k+1} - \bm{x}_{k+1}) \in \partial_{\bm{y}} \mathcal{L}_m^{k+1}.$ Then it follows from the $\bm{\lambda}$-update in \cref{eq:admm_z} that $\bm{d}_2^{k+1} = \bm{\lambda}_k - \bm{\lambda}_{k+1}$. Set $\bm{d}^{k+1} := (\bm{d}_1^{k+1}, \bm{d}_2^{k+1})^\top$. Then $\bm{d}^{k+1} \in \partial \mathcal{L}_m^{k+1}$ and so there exists a positive constant $\tilde{\kappa}_2>0$ such that
	\begin{equation*}
		\begin{aligned}
			\|\bm{d}^{k+1}\|_2^2 & \overset{a}{\leq} 2\|\bm{\lambda}_{k+1} - \bm{\lambda}_{k}\|_2^2 + \rho^2 \| \bm{y}_k - \bm{y}_{k+1}\|_2^2 + 2\rho  \| \bm{\lambda}_{k+1} - \bm{\lambda}_{k}\| \|\bm{y}_{k+1} - \bm{y}_{k} \| \\
			& \overset{b}{\leq}  \left(L_g + \left(L_g+\rho\right)^2\right)\| \bm{y}_{k} - \bm{y}_{k+1}\|_2^2 = \tilde{\kappa}_2\| \bm{y}_{k} - \bm{y}_{k+1}\|_2^2,
		\end{aligned}
	\end{equation*}
	where $\tilde{\kappa}_2 = L_g + \left(L_g+\rho\right)^2>0$, $\overset{a}{\leq}$ is obtained by the Cauchy-Schwarz inequality, and  $\overset{b}{\leq}$ is obtained by $\|\bm{\lambda}_{k+1} - \bm{\lambda}_{k}\|_2 = \|\nabla g(\bm{y}_{k+1}) - \nabla g(\bm{y}_k)\|_2 \leq L_g \|\bm{y}_{k+1} - \bm{y}_k\|_2.$ 
\end{proof}

\begin{lemma}\label{lem:unified_power_diff}
	Let $a,b>0$ and $\bar{p}>0$. Denote $\tilde{m}:=\min\{a,b\}$ and $\tilde{M}:=\max\{a,b\}$. Then
	\begin{equation*}
		\lvert a^{\bar{p}}-b^{\bar{p}}\rvert \;\le\; 
		\begin{cases}
			\bar{p}\,\lvert a-b\rvert	\tilde{m}^{\,\bar{p}-1}, & 0<\bar{p}\le 1,\\[2mm]
			\bar{p}\,\lvert a-b\rvert	\tilde{M}^{\,\bar{p}-1}, & \bar{p}\ge 1.
		\end{cases}
	\end{equation*}
	In particular, for $\bar{p}=1$ both bounds coincide and yield $\lvert a-b\rvert$.
\end{lemma}

\begin{proof}
	Consider $f_{\bar{p}}(x)=x^{\bar{p}}$ on $(0,\infty)$, so $f_{\bar{p}}'(x)=\bar{p}\,x^{\bar{p}-1}$. Without loss of generality assume $a\ge b>0$, and set $\tilde{m}=b$, $\tilde{M}=a$. By the mean value theorem there exists $c\in(b,a)\subseteq[\tilde{m},\tilde{M}]$ such that	$a^{\bar{p}}-b^{\bar{p}}=f_{\bar{p}}'(c)\,(a-b)=\bar{p}\,c^{\,\bar{p}-1}(a-b).$ Taking absolute values gives $\lvert a^{\bar{p}}-b^{\bar{p}}\rvert = \bar{p}\,c^{\,\bar{p}-1}\,\lvert a-b\rvert.$ If $0<\bar{p}\le 1$, then $\bar{p}-1\le 0$ and the function $x\mapsto x^{\bar{p}-1}$ is strictly decreasing on $(0,\infty)$. Hence $c^{\,\bar{p}-1}\le \tilde{m}^{\,\bar{p}-1}$ and
	$\lvert a^{\bar{p}}-b^{\bar{p}}\rvert \le \bar{p}\,\tilde{m}^{\,\bar{p}-1}\,\lvert a-b\rvert$. If $\bar{p}\ge 1$, then $\bar{p}-1\ge 0$ and $x\mapsto x^{\bar{p}-1}$ is (strictly) increasing. This shows that
	$c^{\,\bar{p}-1}\le \tilde{M}^{\,\bar{p}-1}$ and so
	$\lvert a^{\bar{p}}-b^{\bar{p}}\rvert \le \bar{p}\,\tilde{M}^{\,\bar{p}-1}\,\lvert a-b\rvert$.
\end{proof}

\begin{theorem}
	Consider the sequence $\{\bm{\theta}_k\}$ produced by the ADMM iteration framework \cref{eq:admm_updates}, with $\bm{\theta}_k = (\bm{x}_k, \bm{y}_k, \bm{\lambda}_k)$. 
	Assume that $\bm{b} \notin \ker(\bm{A}^{\top})$, $\lambda_{\min}(\bm{A}^\top \bm{A}) > \tfrac{L_g^3 + 2\rho L_g^2 - \rho^3}{\rho^2}$, and that the primal sequence $\{\bm{x}_k\}$ remains bounded. Under these conditions, the trajectory $\{\bm{\theta}_k\}$ possesses finite length, namely $	\sum_{k=1}^\infty \|\bm{\theta}_{k+1} - \bm{\theta}_k\|_2 < \infty,$ which in turn implies that $\bm{\theta}_k$ converges to a stationary solution of \cref{eq:ConstrainedFormulation}.
\end{theorem}

\begin{proof}
	We first establish the subanalyticity of the fractional term. The scalar function $\varphi(x)=|x|^{1/2}$ is subanalytic, and thus the numerator $|\bm{x}|_{1/2}^{1/2}=\sum_i |x_i|^{1/2}$ is also subanalytic. The denominator $|\bm{x}|_2^{1/2}=(\sum_i x_i^2)^{1/4}$, being the composition of $|\bm{x}|_2$ and $r \mapsto r^{1/2}$, is subanalytic on $\mathbb{R}^n \setminus {\bm{0}}$. Hence, their ratio is subanalytic on $\mathbb{R}^n \setminus {\bm{0}}$. Since the quadratic term $\tfrac{1}{2}|\bm{Ax}-\bm{b}|_2^2$ is analytic, the merit function $\mathcal{L}_m(\bm{x},\bm{y})$ is subanalytic.
	
	Next, we verify an important claim: any accumulation point of the sequence $\{(\bm{x}_k,\bm{y}_k)\}$ generated by \cref{alg:admm_full} cannot be $\bm{0}$. Suppose that there exists a subsequence $\bm{x}_{k_j} \to \bm{0}$. By \cref{thm:admm_convergence}, we have $\|\bm{x}_k - \bm{x}_{k+1}\|\to \bm{0}, k \to \infty,$ which implies that $\bm{x}_{k_{j+1}} \to \bm{0}$ as well. Since $\bm{x}_{k_j}-\bm{y}_{k_j}\to \bm{0}$, it follows that $\bm{y}_{k_j}\to \bm{0}$. Consider the subproblem \cref{eq:admm_x}. On the support set $T_{k_j} \neq \emptyset$ it is differentiable, with gradient
	\begin{equation*}
		\nabla_{x^i}\!\left(\frac{\|\bm{x}\|_{1/2}^{1/2}}{\|\bm{x}\|_2^{1/2}}\right) 
		+ \rho\left(x^i - y^i_{k_j} - \frac{\lambda^i_{k_j}}{\rho}\right), 
		\forall i \in T_{k_j},
	\end{equation*}
	where
$\nabla_{x^i}\!\left(\frac{\|\bm{x}\|_{1/2}^{1/2}}{\|\bm{x}\|_2^{1/2}}\right) = \frac{1}{2}\left(\frac{\mathrm{sgn}(x^i)}{\sqrt{|x^i|}\,\|\bm{x}\|_2^{1/2}} - \frac{\|\bm{x}\|_{1/2}^{1/2}\,x^i}{\|\bm{x}\|_2^{5/2}}\right)$. By \cref{lem:subgradient_bound}, any accumulation point $(\bm{x}^*,\bm{y}^*)$ of the sequence generated by the iterative framework \cref{eq:admm_updates} satisfies $\bm{0} \in \partial \mathcal{L}_m(\bm{x}^*,\bm{y}^*)$ and $\bm{x}^*=\bm{y}^*$,	which implies that each accumulation point is a stationary point. Hence, the first-order optimality condition yields $ \lim\limits_{j\to\infty}\nabla_{x^i_{k_j}}\left(\frac{\|\bm{x}\|_{1/2}^{1/2}}{\|\bm{x}\|_2^{1/2}}\right)= \lim\limits_{j\to\infty}\rho\left(y^i_{k_j} - x^i_{k_j} +\frac{\lambda^i_{k_j}}{\rho}\right).$ However, as $j\to\infty$ with $\bm{x}_{k_j}\to\bm{0}$, the gradient term on the left diverges to infinity, while the right-hand side converges to $\lambda_i^*/\rho$ ($\lambda_i^*$ is the $i$-th entry of $\bm{\lambda}^*$). Since \cref{thm:admm_convergence} ensures that $\{\bm{\lambda}_{k_j}\}$ is bounded, this is a contradiction. Hence, no accumulation point of $\{(\bm{x}_k,\bm{y}_k)\}$ can be $\bm{0}$.
	
	By the claim, there exists $\epsilon_{1}>0$ such that the accumulation point $\bm{x}^*$ does not lie in the open hypercube $S_{\epsilon_{1}} := \{\bm{x}\in\mathbb{R}^n : |x_i| < \epsilon_{1}\}.$ Similarly, there exists $\epsilon_{2}>0$ such that $\bm{y}^*$ does not lie in $S_{\epsilon_{2}}$. Let $\epsilon = \max\{\epsilon_{1},\epsilon_{2}\}$. Then the function $\mathcal{L}_m(\bm{x},\bm{y})$ restricted to $S_\epsilon^c \times S_\epsilon^c$ is continuous and closed, and thus satisfies the K\L{} property. Naturally, $\mathcal{L}_m(\bm{x},\bm{y})$ satisfies the K\L{} property at $(\bm{x}^*,\bm{y}^*)$. The remaining part of the proof follows the same arguments as in \cite[Theorem 5.8]{tao2022minimization} and \cite[Theorem 4]{li2015global}, and is omitted here.
\end{proof}

\begin{theorem} 
	Assume that the sensing matrix $\bm{A} \in \mathbb{R}^{m \times n}$ has full row rank, the sequence $\{\bm{x}^k\}$ is bounded, $\lambda_{\min}(\bm{A}^\top \bm{A}) > \frac{L_g^3 + 2\rho L_g^2 - \rho^3}{\rho^2}$, and $\bm{b} \notin \ker(\bm{A}^\top)$. 
	Let $\{\bm{\theta}_k\}$ be the sequence generated by \cref{eq:admm_updates},  $T := \mathrm{supp}(\bm{x}^*)$ be the support of $\bm{x}^*$, and $\Omega:= \left\{ \bm{x} \in \mathbb{R}^n \,\big|\, \mathrm{supp}(\bm{x}) = T,\; \mathrm{sgn}(\bm{x}) = \mathrm{sgn}(\bm{x}^*) \right\}$.
	Then the following statements hold:
	\begin{enumerate}
		\item[(i)] There exists an iteration index $\tilde{K}$ such that $\bm{x}_k, \bm{y}_k \in \Omega$ for all $k \geq \tilde{K}$;
		\item[(ii)] If $\bm{A}_{T}^\top \bm{A}_{T} \succ 0$, then there exists a constant $\bar{\zeta} > 0$ such that $\{\bm{\theta}_k\}$ converges linearly for any $0 < \zeta < \bar{\zeta}$. 
		In particular, there exist constants $\tilde{\kappa}_3 > 0$ and $\varrho \in (0,1)$ such that $\|\bm{\theta}_k - \bm{\theta}^*\| \leq \tilde{\kappa}_3 \varrho^k$ holds.
	\end{enumerate}
\end{theorem}
\begin{proof}
	(i) The proof follows same arguments as in \cite{tao2022minimization}, and is omitted here.
	
	(ii) We adopt the proof framework in \cite{tao2022minimization}. First, we establish an important property of the stationary point $\bm{x}^*$: if $\bm{x}^*$ is a stationary point of problem \cref{eq:ConstrainedFormulation}, then there exists a constant $\bar{\zeta} > 0$ such that, for any $0 < \zeta < \bar{\zeta}$, the objective function $\mathcal{H}$ satisfies the K\L{} property with an exponent of $1/2$ restricted on $\Omega$. To verify the above assertion, we introduce a function $\bar{\mathcal{H}}_T: \mathbb{R}^\alpha \to \mathbb{R}$ ($\alpha = |T|$), defined as $\bar{\mathcal{H}}_T(\bm{u}) := \zeta \frac{\|\bm{u}\|_{1/2}^{1/2}}{\|\bm{u}\|_2^{1/2}} + \frac{1}{2} \|\bm{A}_T \bm{u} - \bm{b}\|_2^2,$
	where $T=\text{supp}(\bm{x}_k)=\text{supp}(\bm{x}^*)$. Define the sets $\Omega_{T}:= \{\bm{u} \in \mathbb{R}^{\alpha} | \bm{u}=\bm{x}_T, \bm{x} \in \Omega\}$ and $\Theta := \{\bm{x} \in \mathbb{R}^n \mid \left(\bm{x}_{T}\right)_i \geq \epsilon, i \in T \}$, where $\epsilon < \omega/n$. To maintain conciseness and avoid repetition, we only prove that
	$ \bar{\mathcal{H}}_T(\bm{u}) $ is Lipschitz continuous on the set $\Omega \cap \Theta$ and there exists a constant $\hat{L}_{\bar{\mathcal{H}}_{T}}>0$ such that, whenever $0 < \zeta < \bar{\zeta}$ and $\bm{x} \in \Omega \cap \Theta$, the inequality $\|\nabla \bar{\mathcal{H}}_T(\bm{x}_T) - \nabla \bar{\mathcal{H}}_T(\bm{x}_T^*)\|_2 \geq \hat{L}_{\bar{\mathcal{H}}_{T}} \|\bm{x}_T - \bm{x}_T^*\|_2$ holds.
	
	Evidently, $\bar{\mathcal{H}}_T(\bm{x}_T) = \mathcal{H}(\bm{x})$ if $\bm{x} \in \Omega$. When $\bm{u} \in \Omega_{T}$,  $\bar{\mathcal{H}}_T$ is differentiable on $\Omega_{T}$, with gradient $\nabla \bar{\mathcal{H}}_T(\bm{u}) = \zeta \left( \frac{|\bm{u}|^{\odot \frac{-1}{2}}  \odot \mathrm{sgn}(\bm{u})}{2 \|\bm{u}\|_2^{1/2}} - \frac{\|\bm{u}\|_{1/2}^{1/2}}{2\|\bm{u}\|_2^{5/2}} \bm{u} \right) + \bm{A}_T^\top (\bm{A}_T \bm{u} - \bm{b}),$ where $|\bm{u}|^{\odot \frac{-1}{2}} = \left(|u_1|^{\frac{-1}{2}},|u_2|^{\frac{-1}{2}}, \cdots,|u_{\alpha}|^{\frac{-1}{2}}\right)^{\top} \in \mathbb{R}^{\alpha}$. It follows directly from the definition of $\Theta$ that $\bm{x}^* \in \Theta$.
	Let $h(\bm{u}) = \nabla \left( \frac{\|\bm{u}\|_{1/2}^{1/2}}{\|\bm{u}\|_2^{1/2}} \right) = \frac{|\bm{u}|^{ \odot \frac{-1}{2}}\cdot \mathrm{sgn}(\bm{u})}{ 2\|\bm{u}\|_2^{1/2}} - \frac{\|\bm{u}\|_{1/2}^{1/2}}{2 \|\bm{u}\|_2^{5/2}} \bm{u}.$ 
	Now we prove that $h(\bm{u})$ is Lipschitz continuous on $\Omega_{T} \cap \Theta_{T}$, where $\Theta_T := \{ \bm{u} \in \mathbb{R}^\alpha \mid \bm{u} = \bm{x}_T, \bm{x} \in \Theta \}$. For any $ \bm{u}, \bm{v} \in \Omega_{T} \cap \Theta_{T} $, one has
	\begin{equation*}
		\begin{aligned}
			\|h(\bm{u}) - h(\bm{v})\| \leq \left\| \frac{|\bm{u}|^{\odot \frac{-1}{2}} \odot \mathrm{sgn}(\bm{u})}{\|\bm{u}\|_2^{1/2}} - \frac{|\bm{v}|^{ \odot \frac{-1}{2}}\odot \mathrm{sgn}(\bm{v})}{ \|\bm{v}\|_2^{1/2}} \right\| + \left\| \frac{\bm{u}\|\bm{u}\|_{1/2}^{1/2}}{ \|\bm{u}\|_2^{5/2}} 
			- \frac{\bm{v}\|\bm{v}\|_{1/2}^{1/2}}{ \|\bm{v}\|_2^{5/2}}  \right\|.
		\end{aligned}
	\end{equation*}
	Decompose $ h(\bm{u}) $ as $ h(\bm{u}) = \Phi(\bm{u}) - \Psi(\bm{u}) $, where $\Phi(\bm{u}) = \frac{|\bm{u}|^{\odot \frac{-1}{2} } \odot \mathrm{sgn}(\bm{u})}{\|\bm{u}\|_2^{\frac{1}{2}}}$ and $\Psi(\bm{u}) = \frac{\|\bm{u}\|_{1/2}^{1/2}}{\|\bm{u}\|_2^{5/2}} \bm{u}$. Note that each pair $ u_i $ and $ v_i $ share the same sign. Thus,
	\begin{equation*}
		\begin{aligned}
			|\Phi_i(\bm{u}) - \Phi_i(\bm{v})| \leq \left| \frac{|u_i|^{1/2} - |v_i|^{1/2}}{|v_i|^{1/2}|u_i|^{1/2}\|\bm{u}\|_2^{1/2}} \right| +  \left| \frac{\|\bm{u}\|_2^{1/2} - \|\bm{v}\|_2^{1/2}}{|v_i|^{1/2}\|\bm{u}\|_2^{1/2}\|\bm{v}\|_2^{1/2}} \right|.
		\end{aligned}
	\end{equation*}
	Since $ |u_i| \geq \epsilon $ and $ |v_i| \geq \epsilon $, combining with \cref{lem:unified_power_diff} leads to
	\begin{equation*}
		\begin{aligned}
			\left| \frac{|u_i|^{1/2} - |v_i|^{1/2}}{|u_i|^{1/2}|u_i|^{1/2}\|\bm{u}\|_2^{1/2}} \right| &\leq \frac{| |v_i| - |v_i| |}{2 \min(|u_i|, |v_i|)^{1/2} } \cdot \frac{1}{\epsilon \|\bm{u}\|_2^{1/2}}\leq \frac{| |u_i| - |v_i| |}{2 \epsilon^{3/2}  \|\bm{u}\|_2^{1/2}} 
		\end{aligned}
	\end{equation*}
	By $\|\bm{u}\|_2 = \left( \sum u_i^2 \right)^{1/2}$ and $|u_i| \geq \epsilon$, we know $\|\bm{u}\|_2^{1/2} \geq (n \epsilon^2)^{1/4} \geq n^{1/4} \epsilon^{1/2}$ and so $\left| \frac{|v_i|^{1/2} - |u_i|^{1/2}}{|u_i|^{1/2} |v_i|^{1/2} \|\bm{u}\|_2^{1/2}} \right| \leq \frac{|u_i - v_i|}{2 n^{1/4} \epsilon^2}.$
	On the other hand, according to \cref{lem:unified_power_diff}, one has
	\begin{equation}
		\begin{aligned}
			\left| \frac{\|\bm{u}\|_2^{1/2} - \|\bm{v}\|_2^{1/2}}{|u_i|^{1/2}\|\bm{u}\|_2^{1/2} \|\bm{v}\|_2^{1/2}} \right| 
			\leq \frac{1}{\epsilon^{3/2} n^{1/2}} \cdot \frac{\left|\|\bm{u}\|_2 - \|\bm{v}\|_2\right|}{2 \min(\|\bm{u}\|_2, \|\bm{v}\|_2)^{1/2}} \leq \frac{\|\bm{u} - \bm{v}\|_2}{2 n^{\frac{3}{4}} \epsilon^2}.
		\end{aligned}
	\end{equation}
	Thus, $\|\Phi(\bm{u}) - \Phi(\bm{v})\|_2 \leq \frac{\sqrt{n}}{2 n^{1/4} \epsilon^2} \|\bm{u} - \bm{v}\|_2 + \frac{n}{2 n^{3/4} \epsilon^2} \|\bm{u} - \bm{v}\|_2= \frac{n^{1/4}}{\epsilon^2} \|\bm{u} - \bm{v}\|_2.$
	Now we show the Lipschitz continuity of $ \Psi(\bm{u})$. In fact,
	\begin{equation}
		\begin{aligned}
			\|\Psi(\bm{u}) - \Psi(\bm{v})\|_2 &\leq \frac{\|\bm{u}\|_{1/2}^{1/2}}{\|\bm{u}\|_2^{5/2}} \|\bm{u} - \bm{v}\|_2 + \left| \frac{\|\bm{u}\|_{1/2}^{1/2}}{\|\bm{u}\|_2^{5/2}} - \frac{\|\bm{v}\|_{1/2}^{1/2}}{\|\bm{v}\|_2^{5/2}} \right| \|\bm{v}\|_2.
		\end{aligned}
	\end{equation}
	By the boundedness of $\|\bm{v}\|_2$, we can assume $\|\bm{v}\|_2 \leq M$. Since $ |u_i| \geq \epsilon $, it follows from \cref{lem:RatioBounds} that $\|\Psi(\bm{u}) - \Psi(\bm{v})\|_2 \overset{a}{\leq} \frac{1}{n^{1/4}\epsilon^2} \|\bm{u} - \bm{v}\|_2 + M\left| \frac{\|\bm{u}\|_{1/2}^{1/2}}{\|\bm{u}\|_2^{5/2}} - \frac{\|\bm{v}\|_{1/2}^{1/2}}{\|\bm{v}\|_2^{5/2}} \right|, $ where the inequality $ \overset{a}{\leq} $ follows from $\|\bm{u}\|_{1/2}^{1/2} \leq n^{3/4} \|\bm{u}\|_2^{1/2}$ and $\frac{\|\bm{u}\|_{1/2}^{1/2}}{\|\bm{u}\|_2^{5/2}} \leq \frac{n^{3/4}}{n \epsilon^2} = \frac{1}{n^{1/4} \epsilon^2}$. Moreover,
	\begin{equation*}
		\begin{aligned}
			&\left| \frac{\|\bm{u}\|_{1/2}^{1/2}}{\|\bm{u}\|_2^{5/2}} - \frac{\|\bm{v}\|_{1/2}^{1/2}}{\|\bm{v}\|_2^{5/2}} \right| 
			\leq \left| \frac{\|\bm{u}\|_{1/2}^{1/2}}{\|\bm{u}\|_2^{5/2}} - \frac{\|\bm{v}\|_{1/2}^{1/2}}{\|\bm{u}\|_2^{5/2}}\right| +  \left| \frac{\|\bm{v}\|_{2}^{5/2}}{\|\bm{u}\|_2^{5/2}\|\bm{v}\|_2^{5/2}} - \frac{\|\bm{u}\|_{2}^{5/2}}{\|\bm{v}\|_2^{5/2}\|\bm{u}\|_2^{5/2}} \right| \|\bm{v}\|_{1/2}^{1/2} \\
			&\leq \frac{1}{\|\bm{u}\|_2^{5/2}} \left| \|\bm{u}\|_{1/2}^{1/2} - \|\bm{v}\|_{1/2}^{1/2} \right| + n^{3/4}\|\bm{v}\|_2^{1/2} \left| \frac{\|\bm{v}\|_{2}^{5/2}}{\|\bm{u}\|_2^{5/2}\|\bm{v}\|_2^{5/2}} - \frac{\|\bm{u}\|_{2}^{5/2}}{\|\bm{v}\|_2^{5/2}\|\bm{u}\|_2^{5/2}} \right|.
		\end{aligned}
	\end{equation*}
	Using the facts $\|\bm{u}\|_2 \leq M$ and $\|\bm{v}\|_2 \leq M$, we have
	\begin{equation}
		\begin{aligned}
			\left| \frac{\|\bm{u}\|_{1/2}^{1/2}}{\|\bm{u}\|_2^{5/2}} - \frac{\|\bm{v}\|_{1/2}^{1/2}}{\|\bm{v}\|_2^{5/2}} \right| 
			& \leq \frac{1}{2n^{5/4} \epsilon^{3}} \sum |u_i - v_i| + \frac{n^{3/4} M^{1/2}}{n^{5/2} \epsilon^5} \cdot \frac{5}{2} M^{3/2} \|\bm{u} - \bm{v}\|_2 \\
			& = \left( \frac{\sqrt{n}}{2 n^{5/4} \epsilon^3} + \frac{5M^2}{2n^{7/4} \epsilon^5} \right) \|\bm{u} - \bm{v}\|_2
		\end{aligned}
	\end{equation}
	and so $\|\Psi(\bm{u}) - \Psi(\bm{v})\|_2 \leq \left( \frac{1}{n^{1/4} \epsilon^2} + M \left( \frac{\sqrt{n}}{2 n^{5/4} \epsilon^3} + \frac{5 M^2}{2 n^{7/4} \epsilon^5} \right) \right) \|\bm{u} - \bm{v}\|_2$.	Thus, combining the Lipschitz constant estimates for $ \Phi(\bm{u}) $ and $ \Psi(\bm{u}) $, one has $\|h(\bm{u}) - h(\bm{v})\|_2 \leq L_h\|\bm{u} - \bm{v}\|_2,$ where $L_h = \frac{1}{2}\left( \frac{n^{1/4}}{\epsilon^2} + \frac{1}{n^{1/4} \epsilon^2} + M \left( \frac{\sqrt{n}}{2 n^{5/4} \epsilon^3} + \frac{5 M^2}{2 n^{7/4} \epsilon^5} \right) \right)$. Consequently, for any $\bm{x} \in \mathcal{E} := \{\bm{x} \in \mathbb{R}^n \mid \bm{x} \in \Omega_{\bm{x}} \cap \mathcal{N}(\bm{x}^*),\;\mathcal{H}(\bm{x}^*) < \mathcal{H}(\bm{x}) < \mathcal{H}(\bm{x}^*) + \varrho\}$, with $\mathcal{N}(\bm{x}^*)$ being a neighborhood of $\bm{x}^*$ such that $\mathcal{N}(\bm{x}^*) \subseteq \Theta$, we have $\text{dist}^2(\bm{0}, \partial \mathcal{H}(\bm{x})) = \inf_{\substack{\bm{d} \in \partial \mathcal{H}(\bm{x})}} \|\bm{d}\|_2^2 \geq \inf_{\bm{d} \in \partial \mathcal{H}(\bm{x})} \|\bm{d}_T\|_2^2 = \|\nabla \bar{\mathcal{H}}_T(\bm{x}_T)\|_2^2$.
	
	In summary, $\mathcal{H}$ satisfies the K\L\ property on $\Omega$ with an exponent of $1/2$. The remaining part of the proof follows along the same lines as \cite[Theorem 5.9]{tao2022minimization}, and is omitted here for brevity.
\end{proof}

\section{Numerical experiments}\label{sec:experiments}
In this section, we provide some numerical experiments to indicate our method consistently outperforms the existing ones.
  
\subsection{Experimental Setup}

In this subsection, consider a series of numerical experiments on a Windows 11 computing platform equipped with a 12th Gen Intel(R) Core i9-12900H processor (2.5 GHz) and 16 GB RAM, to evaluate the performance of $\ell_{1/2}/\ell_2$ and its ADMM algorithmic framework for sparse recovery. All algorithms are implemented in MATLAB and executed under version R2022b. The proposed $\ell_{1/2}/\ell_2$ is benchmarked against several state-of-the-art approach, including $\ell_1$, $\ell_{1/2}$, $\text{IRLS}\ell_{p}$ \cite{lai2013improved}, $\ell_1-\alpha\ell_2$, $\ell_1/\ell_2$, and $\ell_1/\ell_\infty$. The sparse recovery performance of $\ell_{1/2}/\ell_2$ is examined under two types of sensing matrices: random Gaussian and oversampled discrete cosine transform (DCT). The random Gaussian sensing matrix is generated from the distribution $\mathcal{N}(0, \bm{\Sigma})$, where $\Sigma = \{(1-r)\bm{I}_n (i = j) + r\}_{i,j}$ . For the oversampled DCT matrix $\bm{A} = [\bm{a}_1, \bm{a}_2, \ldots, \bm{a}_n] \in \mathbb{R}^{m \times n}$, the $i$-th column vector $\bm{a}_i$ is defined as $\bm{a}_i = \frac{1}{\sqrt{m}} \cos\left( \frac{2i\pi\bm{\omega}}{F} \right),$ where $F > 0$ denotes the coherence control parameter. The vector $\bm{\omega} \in \mathbb{R}^m$ is drawn from the uniform distribution $\mathcal{U}([0,1]^m)$. 

The initial points are chosen as follows: the zero vector for $\ell_1$; the $\ell_1$ solution for $\ell_1-\alpha\ell_2$ (recommended by \cite{lou2018fast}); the pseudoinverse solution for $\text{IRLS}\ell_{p}$ (recommended by \cite{lai2013improved}); and the DCA solution of $\ell_1-\ell_2$ for the remaining models. The regularization parameter $\zeta$ is fixed at $10^{-5}$ for all models, except for $\ell_{1/2}$, in which it is determined adaptively. The stopping criterion for all algorithms is set to $\frac{\|\bm{x}_k - \bm{x}_{k-1}\|_2}{\|\bm{x}_k\|_2} < 10^{-8}$, with a maximum of $5n$ iterations allowed. All other parameters are set to the default values for the corresponding models. For $\text{IRLS}\ell_{p}$, the parameter $p$ is fixed at $1/2$. The sensing matrix $\bm{A}$ is fixed at a dimension of $64 \times 512$. Additionally, we employ the strategies from \cite{boyd2011distributed} and \cite{ding2019alphaL1} to dynamically adjust the penalty parameters $\rho$ and $\gamma$ during the iterations. Let $\bm{x}^*$ denote the reconstructed sparse solution for the ground-truth signal $\bm{x}$. Following the evaluation strategy in \cite{rahimi2019scale}, we employ three metrics: success rate, model failure rate, and algorithm failure rate. A recovery is considered successful when $\frac{\|\bm{x} - \bm{x}^*\|_2}{\|\bm{x}\|_2} \leq 10^{-3}$. If $\mathcal{H}(\bm{x}^*) > \mathcal{H}(\bm{x})$, the trial is classified as an algorithm failure. Conversely, if $\mathcal{H}(\bm{x}^*) < \mathcal{H}(\bm{x})$, it is classified as a model failure.

\subsection{Recovery of Noiseless Sparse Vectors}
In this subsection, consider a systematic evaluation of the recovery performance of various sparse optimization models under two representative test scenarios: (1) the effect of different $F\in \{5, 10, 15, 20\}$ with oversampled DCT matrices , and (2) the effect of different parameters $r\in \{0.2, 0.4, 0.6, 0.8\}$ with Gaussian matrices. In \cref{fig:DCT_FSuccessAggorithmModel} and \cref{fig:Gauss_RSuccessAggorithmModel}, the notation $\mathrm{ADMM}^{\mathrm{weighted}}$ indicates that the weight parameter $\alpha$ is updated dynamically. Similarly, the $\mathrm{ADMM}_w$ and $\mathrm{FB}_w$ in Subsection \ref{sec:noise} adopt the same adaptive strategy. For each condition, every algorithm is independently executed 50 times. The experimental results show that $\ell_{1/2}/\ell_2$ consistently exhibits superior sparsity-promoting capability, robustness, and stability across different scenarios. In the oversampled DCT case, the model achieves a significantly higher success rate than the methods mentioned above in the low-to-moderate sparsity range ($\text{Sparsity} < 20$), displays a much slower decline as sparsity increases, and maintains an excellent performance even when $F = 20$. Moreover, our model and algorithm failure rates remain low across different $F$ values, which demonstrates low sensitivity to $F$ variation. Under Gaussian matrices, the performance remains similarly stable across different $r$ values. Even at $r = 0.8$, our model sustains a high success rate, with failure rates increasing only slightly. Overall, the $\ell_{1/2}/\ell_2$ outperforms all the methods mentioned above in the sparsity-promoting and the adaptability as well as the robustness under different matrix structures and parameter variations. 
\begin{figure}[!htbp]
	\centering
	\setlength{\abovecaptionskip}{0pt}    
	\includegraphics[scale=0.25]{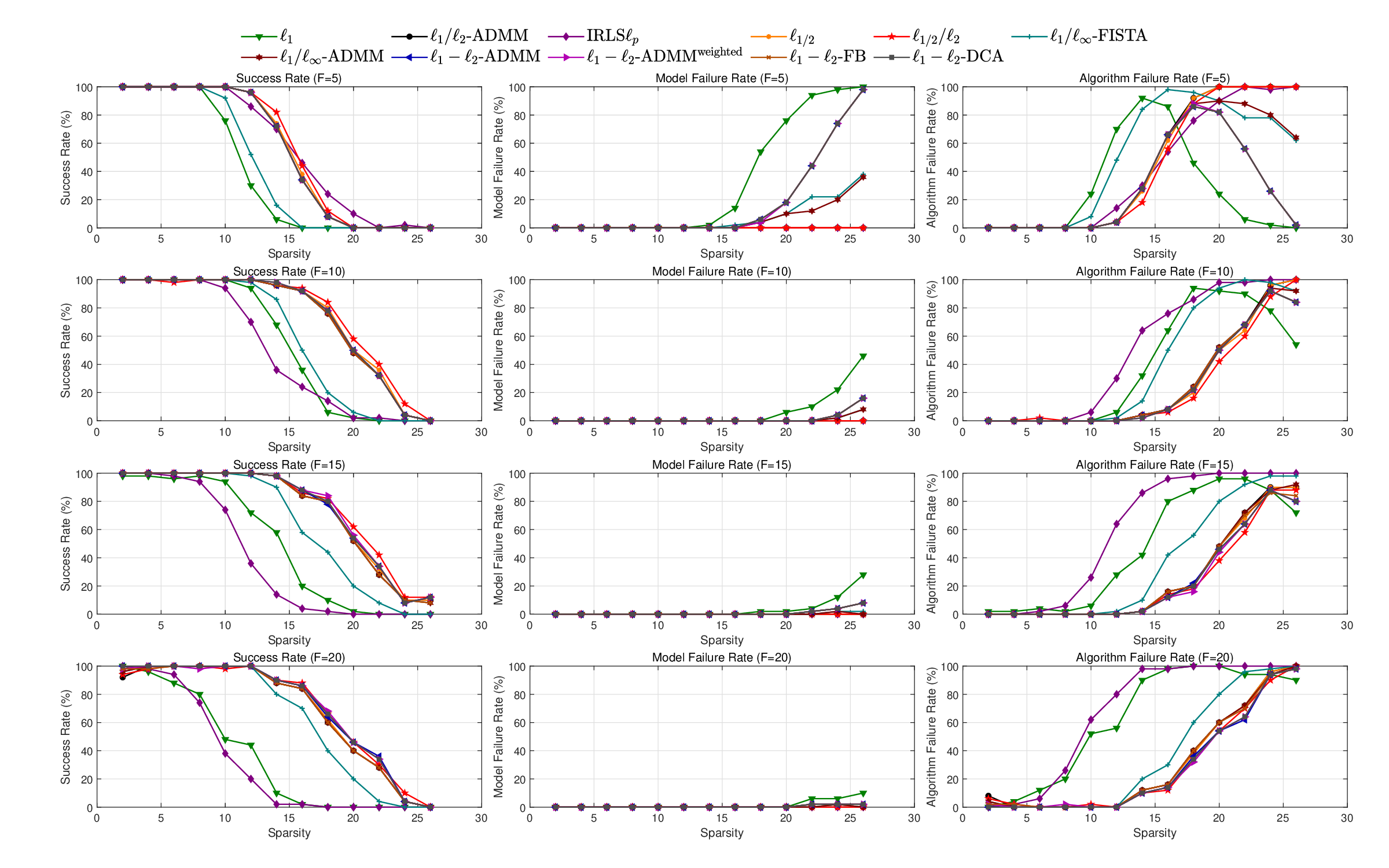} 
	\caption{Comparison of success rate, algorithm failure rate, and model failure rate of various models under different $F$ values with oversampled DCT matrices.}
	\label{fig:DCT_FSuccessAggorithmModel}
\end{figure}

\begin{figure}[!htbp]
	\centering
	\setlength{\abovecaptionskip}{0pt}  
	\includegraphics[scale=0.25]{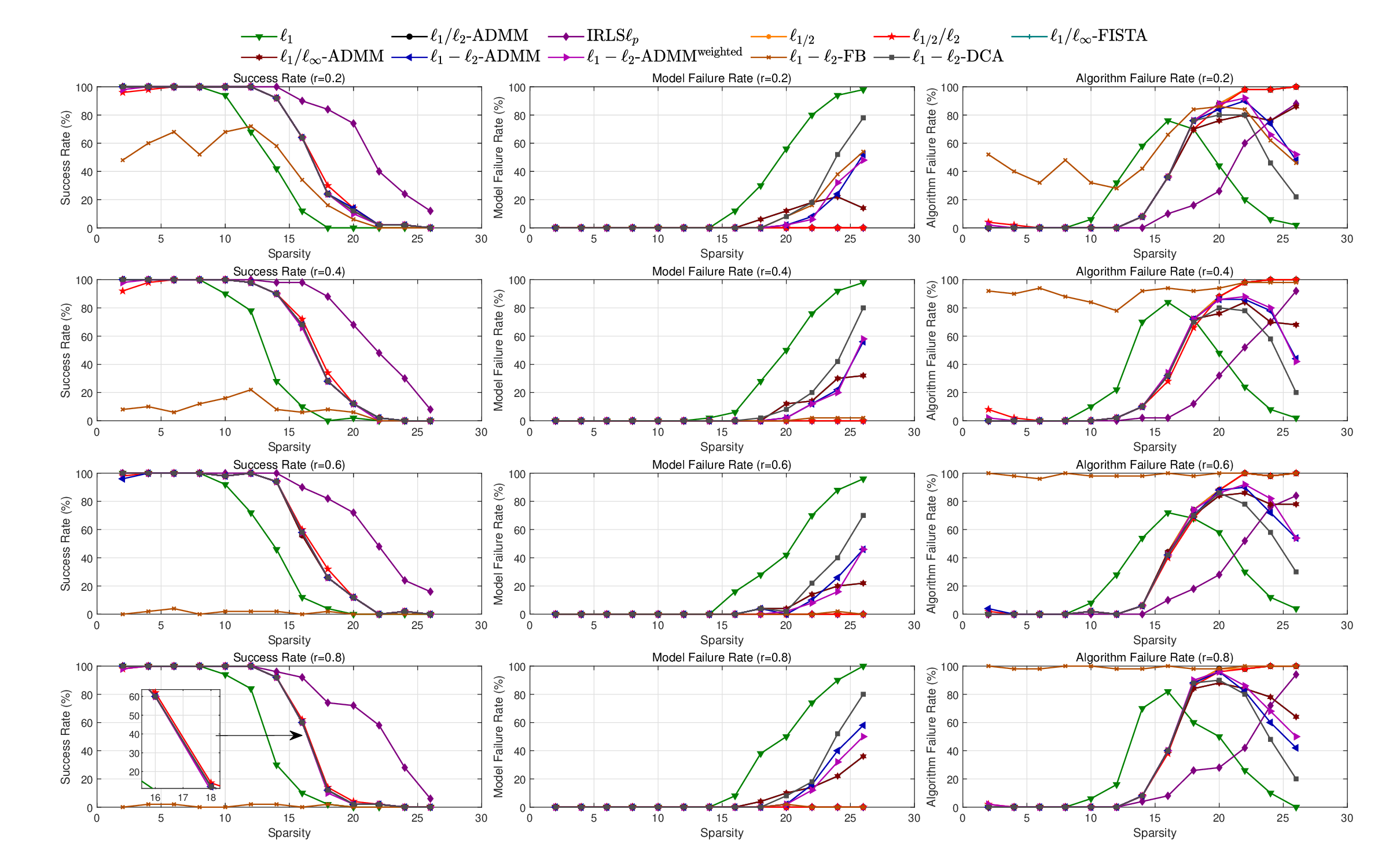} 
	\caption{Comparison of success rate, algorithm failure rate, and model failure rate of various models under different $r$ values with Gaussian matrices.}
	\label{fig:Gauss_RSuccessAggorithmModel}
\end{figure}

\subsection{Recovery of Noise Sparse Vectors}\label{sec:noise}
In this subsection, we evaluate the sparse recovery performance of the $\ell_{1/2}/\ell_2$ model under noisy conditions. The quality of recovery is measured by the signal-to-noise ratio (SNR, unit: dB) of reconstructed solution $\bm{x}^{\ast}$, defined as $\mathrm{SNR}(\bm{x}, \bm{x}^*) = 10 \log_{10} \left( \frac{\|\bm{x}^{\ast} - \bm{x}\|_2^2}{\|\bm{x}\|_2^2} \right).$ This strategy has also been adopted in the work of \cite{yin2015minimization}. For the oversampled DCT matrix, the sparsity level is set to $s = 15$, the minimum separation is $L = 15$, $F \in \{5, 10, 15, 20\}$, and $\zeta$ is fixed at $8 \times 10^{-4}$. For the Gaussian matrix, the sparsity is set to $s = 10$, the minimum separation is $L = 15$, $r \in \{0.2, 0.4, 0.6, 0.8\}$, and $\zeta$ is set to $8 \times 10^{-3}$.  Each experimental configuration is repeated 50 times. Considering the superior performance of the IRLS$\ell_p$ model under Gaussian measurements, for all models except $\ell_1$ and IRLS$\ell_p$, the initial point is chosen as the solution obtained by the IRLS$\ell_p$. For the oversampled DCT matrix, the initial point is selected as the solution from the DCA algorithm applied to the $\ell_1$-$\ell_2$. For the $\ell_{1/2}$ and IRLS$\ell_p$, the true sparsity $s$ is used as the estimated sparsity. 
\begin{table}[!htbp]
	\centering
	\caption{Average rankings of different sparse optimization models under over-sampled DCT and Gaussian matrices with various parameters. The rankings are obtained from 50 independent trials for each setting.}
	\label{tab:ranking}
	\resizebox{\linewidth}{!}{
		\begin{tabular}{cccccccccccccc} 
			\toprule
			& \begin{tabular}[c]{@{}c@{}}Over Sampled \\ DCT matrix\end{tabular} & $\ell_1$     & IRLS$\ell_{p}$   & \begin{tabular}[c]{@{}c@{}}$\ell_1/\ell_{2}$ \\ ADMM\end{tabular} & $\ell_p$    & \begin{tabular}[c]{@{}c@{}}$\ell_{1}/\ell_{\infty}$ \\ FISTA \end{tabular}  & \begin{tabular}[c]{@{}c@{}}$\ell_{1}/\ell_{\infty}$ \\ ADMM \end{tabular} & $\ell_{1/2}/\ell_{2}$ & \begin{tabular}[c]{@{}c@{}}$\ell_{1}-\alpha\ell_{2}$ \\ ADMM\end{tabular}  & \begin{tabular}[c]{@{}c@{}}$\ell_{1}-\alpha\ell_{2}$ \\ $\text{ADMM}_w$\end{tabular} & \begin{tabular}[c]{@{}c@{}}$\ell_{1}-\alpha\ell_{2}$ \\ FB \end{tabular} & \begin{tabular}[c]{@{}c@{}}$\ell_{1}-\ell_{2}$ \\ DCA \end{tabular} & \begin{tabular}[c]{@{}c@{}}$\ell_{1}-\alpha\ell_{2}$ \\ $\text{FB}_w$ \end{tabular} \\ 
			\hline
			\multirow{4}{*}{\begin{tabular}[c]{@{}c@{}}$s=15$\\$L=15$\\$\text{SNR}=45 \text{dB}$\end{tabular}} & $F=5$              & 11.08  & 4.42   & 8.22        & 3.7   & 9.88              & 7.34             & \textbf{2.24}      & 4.52      & 4.82              & 5.32    & 5.6      & 10.86            \\
			& $F=10$             & 10.42  & 11.42  & 7.36        & 4.02  & 9.4               & 6.78             & \textbf{1.56}      & 3.62      & 4.04              & 4.56    & 5.04     & 9.78             \\
			& $F=15$             & 10.3   & 12.00~ & 6.34        & 5.32  & 8.76              & 6.00~            & \textbf{3.04}      & 3.54      & 4.18              & 4.56    & 4.8      & 9.16             \\
			& $F=20$             & 9.78   & 12.00~ & 5.34        & 6.48  & 7.84              & 5.56             & \textbf{3.94}      & 4.32      & 4.64              & 4.8     & 4.8      & 8.5              \\ 
			\cline{2-14}
			\multirow{4}{*}{\begin{tabular}[c]{@{}c@{}}$s=15$\\$L=15$\\$\text{SNR}=50 \text{dB}$\end{tabular}} & $F=5$              & 10.88~ & 4.04~  & 8.16~       & 3.26~ & 10.14~            & 7.74~            & \textbf{2.04}~     & 4.70~     & 5.32~             & 5.34~   & 5.54~    & 10.84~           \\
			& $F=10$             & 10.16~ & 11.54~ & 7.68~       & 3.44~ & 9.02~             & 6.96~            & \textbf{1.16}~     & 4.10~     & 4.20~             & 4.48~   & 4.72~    & 10.54~           \\
			& $F=15$             & 9.80~  & 11.82~ & 6.86~       & 5.14~ & 7.96~             & 6.10~            & \textbf{2.30}~     & 4.22~     & 4.52~             & 4.68~   & 4.72~    & 9.88~            \\
			& $F=20$             & 9.52~  & 12.00~ & 5.60~       & 6.14~ & 6.90~             & 5.60~            & \textbf{3.46}~     & 4.32~     & 4.74~             & 5.00~   & 5.18~    & 9.54~            \\ 
			\cline{2-14}
			\multirow{4}{*}{\begin{tabular}[c]{@{}c@{}}$s=15$\\$L=15$\\$\text{SNR}=55 \text{dB}$\end{tabular}} & $F=5$              & 10.78~ & 4.04~  & 8.22~       & 3.30~ & 10.12~            & 7.94~            & \textbf{1.48}~     & 5.16~     & 5.24~             & 5.46~   & 5.48~    & 10.78~           \\
			& $F=10$             & 10.12~ & 11.56~ & 7.56~       & 3.34~ & 8.98~             & 6.90~            & \textbf{1.04}~     & 4.14~     & 4.22~             & 4.76~   & 4.68~    & 10.70~           \\
			& $F=15$             & 10.04~ & 11.82~ & 6.70~       & 5.00~ & 7.78~             & 6.18~            & \textbf{2.42}~     & 4.32~     & 4.46~             & 4.36~   & 4.64~    & 10.28~           \\
			& $F=20$             & 9.26~  & 11.78~ & 5.56~       & 5.78~ & 6.78~             & 5.42~            & \textbf{3.82}~     & 4.42~     & 4.82~             & 5.12~   & 5.28~    & 9.96~            \\ 
			\hline
			& Gaussian matrix         & $\ell_1$     & IRLS$\ell_{p}$   & \begin{tabular}[c]{@{}c@{}}$\ell_1/\ell_{2}$ \\ ADMM\end{tabular} & $\ell_p$    & \begin{tabular}[c]{@{}c@{}}$\ell_{1}/\ell_{\infty}$ \\ FISTA \end{tabular}  & \begin{tabular}[c]{@{}c@{}}$\ell_{1}/\ell_{\infty}$ \\ ADMM \end{tabular} & $\ell_{1/2}/\ell_{2}$ & \begin{tabular}[c]{@{}c@{}}$\ell_{1}-\alpha\ell_{2}$ \\ ADMM\end{tabular}  & \begin{tabular}[c]{@{}c@{}}$\ell_{1}-\alpha\ell_{2}$ \\ $\text{ADMM}_w$\end{tabular} & \begin{tabular}[c]{@{}c@{}}$\ell_{1}-\alpha\ell_{2}$ \\ FB \end{tabular} & \begin{tabular}[c]{@{}c@{}}$\ell_{1}-\ell_{2}$ \\ DCA \end{tabular} & \begin{tabular}[c]{@{}c@{}}$\ell_{1}-\alpha\ell_{2}$ \\ $\text{FB}_w$ \end{tabular}  \\ 
			\hline
			\multirow{4}{*}{\begin{tabular}[c]{@{}c@{}}$s=10$\\$L=15$\\$\text{SNR}=45 \text{dB}$\end{tabular}} & $r=0.2$            & 5.68~  & \textbf{1.70}~  & 4.40~       & 5.36~ & 9.86~             & 3.14~            & \textbf{1.70}~     & 7.68~     & 7.78~             & 11.00~  & 7.70~    & 12.00~           \\
			& $r=0.4$            & 6.80~  & 1.66~  & 4.04~       & 5.16~ & 9.84~             & 2.96~            & \textbf{1.60}~     & 7.56~     & 7.80~             & 11.00~  & 7.58~    & 12.00~           \\
			& $r=0.6$            & 6.80~  & \textbf{1.56}~  & 4.16~       & 5.12~ & 9.78~             & 2.80~            & 1.92~     & 7.58~     & 7.78~             & 11.00~  & 7.50~    & 12.00~           \\
			& $r=0.8$            & 7.32~  & 2.44~  & 4.22~       & 4.52~ & 9.66~             & 2.32~            & \textbf{1.84}~     & 7.58~     & 7.54~             & 11.00~  & 7.56~    & 12.00~           \\ 
			\cline{2-14}
			\multirow{4}{*}{\begin{tabular}[c]{@{}c@{}}$s=10$\\$L=15$\\$\text{SNR}=50 \text{dB}$\end{tabular}} & $r=0.2$            & 5.34~  & 1.96~  & 4.66~       & 5.18~ & 9.76~             & 3.28~            & \textbf{1.86}~     & 7.60~     & 7.70~             & 11.00~  & 7.66~    & 12.00~           \\
			& $r=0.4$            & 6.52~  & \textbf{1.82}~  & 4.20~       & 4.96~ & 9.72~             & 2.96~            & 2.26~     & 7.52~     & 7.64~             & 11.00~  & 7.40~    & 12.00~           \\
			& $r=0.6$            & 6.68~  & \textbf{2.20}~  & 4.14~       & 5.04~ & 9.72~             & 2.50~            & 2.52~     & 7.30~     & 7.42~             & 11.00~  & 7.48~    & 12.00~           \\
			& $r=0.8$            & 7.06~  & 2.98~  & 4.16~       & 4.12~ & 9.76~             & \textbf{2.30}~            & 2.40~     & 7.18~     & 7.68~             & 11.00~  & 7.36~    & 12.00~           \\ 
			\cline{2-14}
			\multirow{4}{*}{\begin{tabular}[c]{@{}c@{}}$s=10$\\$L=15$\\$\text{SNR}=55 \text{dB}$\end{tabular}} & $r=0.2$            & 5.54~  & 2.24~  & 5.06~       & 4.84~ & 9.62~             & 3.18~            & \textbf{2.08}~     & 7.64~     & 7.36~             & 11.00~  & 7.44~    & 12.00~           \\
			& $r=0.4$            & 6.40~  & \textbf{2.12}~  & 4.50~       & 4.64~ & 9.58~             & 2.68~            & 2.94~     & 7.46~     & 7.24~             & 11.00~  & 7.44~    & 12.00~           \\
			& $r=0.6$            & 6.70~  & 2.82~  & 4.34~       & 4.32~ & 9.70~             & \textbf{2.22}~            & 2.48~     & 7.50~     & 7.42~             & 11.00~  & 7.50~    & 12.00~           \\
			& $r=0.8$            & 7.94~  & 3.96~  & 4.46~       & 3.22~ & 9.44~             & 3.08~            & \textbf{2.86}~     & 6.60~     & 6.60~             & 11.00~  & 6.84~    & 12.00~           \\
			\bottomrule
	\end{tabular}}
\end{table}

\cref{tab:ranking}, which reports the average rankings, $\ell_{1/2}/\ell_{2}$ consistently achieves lower ranks across most experimental configurations, particularly under the over-sampled DCT matrix setting. For instance, when the SNR are 45 dB and 50 dB, $\ell_{1/2}/\ell_{2}$ frequently ranks among the top two across different oversampling factors $F$, which demonstrates the stronger robustness against noise. In the case of Gaussian matrices, although the overall differences between methods are smaller, $\ell_{1/2}/\ell_{2}$ still achieves leading positions at multiple sparsity levels, which indicates the better adaptability to varying measurement structures.  

\cref{tab:winning}, which summarizes the winning counts, further confirms the superiority of $\ell_{1/2}/\ell_{2}$. The term ``winning'' means that, the winner is $\ell_{1/2}/\ell_{2}$ if $\text{SNR}_{\ell_{1/2}/\ell_{2}} < \text{SNR}_{M_c}$, and $M_c$ otherwise, where $M_c$ denotes other model mentioned above. Since the SNR reflects reconstruction error, a lower SNR value corresponds to a smaller relative reconstruction error and thus indicates better recovery performance. Across various experimental settings, our model substantially outperforms most competing methods. Under the over-sampled DCT matrix, $\ell_{1/2}/\ell_{2}$ attains significantly higher winning counts, especially at medium-to-high SNR levels (50 dB and 55 dB), highlighting its strong sparse recovery capability in low-noise scenarios. In the Gaussian matrix case, the model maintains consistently competitive winning counts.  

Overall, the results from both average ranking and winning counts demonstrate the consistent advantage of $\ell_{1/2}/\ell_{2}$. These findings indicate that $\ell_{1/2}/\ell_{2}$ not only achieves superior reconstruction accuracy but also maintains stability across different noise levels and measurement structures.
 
\begin{table}
	\centering
	\caption{Winning counts of different sparse optimization models under over-sampled DCT and Gaussian matrices with various parameters.}
	\label{tab:winning}
	\resizebox{\linewidth}{!}{
		\begin{tabular}{cccccccccccccc} 
			\toprule
			& \begin{tabular}[c]{@{}c@{}}Over Sampled \\ DCT matrix\end{tabular} & L1     & IRLS$\ell_{p}$   & \begin{tabular}[c]{@{}c@{}}$\ell_1/\ell_{2}$ \\ ADMM\end{tabular} & $\ell_p$    & \begin{tabular}[c]{@{}c@{}}$\ell_{1}/\ell_{\infty}$ \\ FISTA \end{tabular}  & \begin{tabular}[c]{@{}c@{}}$\ell_{1}/\ell_{\infty}$ \\ ADMM \end{tabular} & $\ell_{1/2}/\ell_{2}$ & \begin{tabular}[c]{@{}c@{}}$\ell_{1}-\alpha\ell_{2}$ \\ ADMM\end{tabular}  & \begin{tabular}[c]{@{}c@{}}$\ell_{1}-\alpha\ell_{2}$ \\ $\text{ADMM}_w$\end{tabular} & \begin{tabular}[c]{@{}c@{}}$\ell_{1}-\alpha\ell_{2}$ \\ FB \end{tabular} & \begin{tabular}[c]{@{}c@{}}$\ell_{1}-\ell_{2}$ \\ DCA \end{tabular} & \begin{tabular}[c]{@{}c@{}}$\ell_{1}-\alpha\ell_{2}$ \\ $\text{FB}_w$ \end{tabular} \\ 
			\hline
			\multirow{4}{*}{\begin{tabular}[c]{@{}c@{}}$s=15$\\$L=15$\\$\text{SNR}=45 \text{dB}$\end{tabular}} & $F=5$                                                        & 0~  & 31~  & 3~          & 5~  & 2~                & 3~               & 488~      & 4~        & 5~                & 4~      & 4~       & 1~               \\
			& $F=10$                                                       & 0~  & 1~   & 3~          & 3~  & 3~                & 3~               & 522~      & 4~        & 3~                & 3~      & 3~       & 2~               \\
			& $F=15$                                                       & 1~  & 0~   & 14~         & 10~ & 5~                & 12~              & 448~      & 15~       & 15~               & 12~     & 12~      & 6~               \\
			& $F=20$                                                       & 4~  & 0~   & 19~         & 11~ & 10~               & 18~              & 403~      & 21~       & 20~               & 17~     & 17~      & 10~              \\ 
			\cline{2-14}
			\multirow{4}{*}{\begin{tabular}[c]{@{}c@{}}$s=15$\\$L=15$\\$\text{SNR}=50 \text{dB}$\end{tabular}} & $F=5$                                                        & 2~  & 19~  & 4~          & 6~  & 1~                & 3~               & 498~      & 4~        & 4~                & 4~      & 4~       & 1~               \\
			& $F=10$                                                       & 0~  & 1~   & 0~          & 1~  & 2~                & 0~               & 542~      & 1~        & 1~                & 1~      & 1~       & 0~               \\
			& $F=15$                                                       & 3~  & 1~   & 7~          & 6~  & 7~                & 7~               & 485~      & 8~        & 8~                & 7~      & 7~       & 4~               \\
			& $F=20$                                                       & 5~  & 0~   & 13~         & 10~ & 14~               & 13~              & 427~      & 16~       & 18~               & 14~     & 14~      & 6~               \\ 
			\cline{2-14}
			\multirow{4}{*}{\begin{tabular}[c]{@{}c@{}}$s=15$\\$L=15$\\$\text{SNR}=55 \text{dB}$\end{tabular}} & $F=5$                                                        & 2~  & 8~   & 1~          & 5~  & 1~                & 1~               & 526~      & 1~        & 1~                & 1~      & 1~       & 2~               \\
			& $F=10$                                                       & 0~  & 0~   & 0~          & 1~  & 1~                & 0~               & 548~      & 0~        & 0~                & 0~      & 0~       & 0~               \\
			& $F=15$                                                       & 1~  & 1~   & 8~          & 7~  & 9~                & 8~               & 479~      & 9~        & 10~               & 8~      & 8~       & 2~               \\
			& $F=20$                                                       & 7~  & 1~   & 16~         & 12~ & 15~               & 16~              & 409~      & 18~       & 20~               & 16~     & 16~      & 4~               \\ 
			\hline
			& Gaussian matrix         & L1     & IRLS$\ell_{p}$   & \begin{tabular}[c]{@{}c@{}}$\ell_1/\ell_{2}$ \\ ADMM\end{tabular} & $\ell_p$    & \begin{tabular}[c]{@{}c@{}}$\ell_{1}/\ell_{\infty}$ \\ FISTA \end{tabular}  & \begin{tabular}[c]{@{}c@{}}$\ell_{1}/\ell_{\infty}$ \\ ADMM \end{tabular} & $\ell_{1/2}/\ell_{2}$ & \begin{tabular}[c]{@{}c@{}}$\ell_{1}-\alpha\ell_{2}$ \\ ADMM\end{tabular}  & \begin{tabular}[c]{@{}c@{}}$\ell_{1}-\alpha\ell_{2}$ \\ $\text{ADMM}_w$\end{tabular} & \begin{tabular}[c]{@{}c@{}}$\ell_{1}-\ell_{2}$ \\ FB \end{tabular} & \begin{tabular}[c]{@{}c@{}}$\ell_{1}-\alpha\ell_{2}$ \\ DCA \end{tabular} & \begin{tabular}[c]{@{}c@{}}$\ell_{1}-\alpha\ell_{2}$ \\ $\text{FB}_w$ \end{tabular} \\ 
			\hline
			\multirow{4}{*}{\begin{tabular}[c]{@{}c@{}}$s=10$\\$L=15$\\$\text{SNR}=45 \text{dB}$\end{tabular}} & $r=0.2$                                                      & 4~  & 17~  & 2~          & 2~  & 2~                & 2~               & 515~      & 2~        & 2~                & 0~      & 2~       & 0~               \\
			& $r=0.4$                                                      & 0~  & 17~  & 2~          & 2~  & 1~                & 2~               & 520~      & 2~        & 2~                & 0~      & 2~       & 0~               \\
			& $r=0.6$                                                      & 2~  & 23~  & 3~          & 3~  & 1~                & 11~              & 504~      & 1~        & 1~                & 0~      & 1~       & 0~               \\
			& $r=0.8$                                                      & 3~  & 12~  & 2~          & 6~  & 1~                & 12~              & 508~      & 2~        & 2~                & 0~      & 2~       & 0~               \\ 
			\cline{2-14}
			\multirow{4}{*}{\begin{tabular}[c]{@{}c@{}}$s=10$\\$L=15$\\$\text{SNR}=50 \text{dB}$\end{tabular}} & $r=0.2$                                                      & 5~  & 15~  & 4~          & 3~  & 3~                & 4~               & 507~      & 3~        & 3~                & 0~      & 3~       & 0~               \\
			& $r=0.4$                                                      & 5~  & 16~  & 6~          & 6~  & 5~                & 7~               & 487~      & 6~        & 6~                & 0~      & 6~       & 0~               \\
			& $r=0.6$                                                      & 5~  & 15~  & 7~          & 7~  & 5~                & 16~              & 474~      & 7~        & 7~                & 0~      & 7~       & 0~               \\
			& $r=0.8$                                                      & 3~  & 13~  & 8~          & 10~ & 5~                & 13~              & 480~      & 6~        & 6~                & 0~      & 6~       & 0~               \\ 
			\cline{2-14}
			\multirow{4}{*}{\begin{tabular}[c]{@{}c@{}}$s=10$\\$L=15$\\$\text{SNR}=55 \text{dB}$\end{tabular}} & $r=0.2$                                                      & 6~  & 11~  & 5~          & 5~  & 5~                & 7~               & 496~      & 5~        & 5~                & 0~      & 5~       & 0~               \\
			& $r=0.4$                                                      & 11~ & 14~  & 10~         & 10~ & 10~               & 12~              & 453~      & 10~       & 10~               & 0~      & 10~      & 0~               \\
			& $r=0.6$                                                      & 6~  & 13~  & 7~          & 7~  & 7~                & 13~              & 476~      & 7~        & 7~                & 0~      & 7~       & 0~               \\
			& $r=0.8$                                                      & 4~  & 9~   & 13~         & 14~ & 8~                & 9~               & 457~      & 12~       & 12~               & 0~      & 12~      & 0~               \\
			\bottomrule
	\end{tabular}}
\end{table}

\subsection{Structural Sparsification of RVFL via $\ell_{1/2}/\ell_2$}

In this subsection, to verify the generalization capability and sparse representation advantage of $\ell_{1/2}/\ell_{2}$ in practical machine learning tasks, we apply it to the structural sparsification of Random Vector Functional Link (RVFL) network.

The RVFL network, first proposed by Pao \textit{et al.} \cite{pao1994learning}, can be regarded as a variant of the single-layer feedforward neural network (SLFN). Given a training dataset $ \mathcal{\bm{D}}=\{(\bm{x}_i, y_i)\}_{i=1}^N, \bm{x}_i \in \mathbb{R}^d, \; y_i \in \mathbb{R}^m, $ the output of the RVFL network for an input $\bm{x}_i$ is expressed as $ \hat{y}_i = \sum_{j=1}^L \beta_j a(\bm{w}_j^\top \bm{x}_i + b_j) + \bm{\tilde{v}}^\top \bm{x}_i, $ where $L$ denotes the number of hidden nodes, $\bm{w}_j \in \mathbb{R}^d$ and $b_j \in \mathbb{R}$ are, respectively, randomly initialized weights and biases that remain fixed during training, $a(\cdot)$ is a nonlinear activation function, $\bm{\beta} \in \mathbb{R}^L$ are the output weights to be learned, and $\bm{\tilde{v}} \in \mathbb{R}^d$ corresponds to the direct link from the input. By stacking the hidden-layer outputs and the direct-link components for all training samples, the enhanced feature matrix is $\bm{H} = \left[a(\bm{X}\bm{W} + \bm{b}\mathbf{1}^\top)  \bm{X} \right]\in \mathbb{R}^{N \times (L+d)}$, where $\bm{X} = [\bm{x}_1, \dots, \bm{x}_N]^\top \in \mathbb{R}^{N \times d}$, $\bm{W} = [\bm{w}_1, \dots, \bm{w}_L] \in \mathbb{R}^{d \times L}$, $\bm{b} = [b_1, \dots, b_L]^\top$. Thus, the training of RVFL can be formulated as $\min_{\bm{\beta}} \; \|\bm{H\beta} - \bm{Y}\|_2^2 + \lambda R(\bm{\beta})$, where $\bm{Y}=[y_1^\top, \ldots, y_N^\top]^\top \in \mathbb{R}^{N \times m}$ is the target output matrix, $\bm{\beta} \in \mathbb{R}^{(L+d)\times m}$ is the output weight, $\lambda>0$ is the regularization parameter, and $R(\cdot)$ is the regularization term. Let $R(\bm{\beta}) = \|\bm{\beta}\|_{1/2}^{1/2}/\|\bm{\beta}\|_2^{1/2},$ which simultaneously preserves scale stability and enforces sparsity in the weight vector. This mechanism automatically identifies and suppresses hidden neurons with negligible contributions, which leads to a more compact sparse RVFL model.

The first five benchmark datasets in \cref{tab:regresssion} are from the OpenML repository\footnote{\url{https://www.openml.org/}} and the remains are from the StatLib archive\footnote{\url{https://lib.stat.cmu.edu/datasets/}}.  The test set mean squared error (MSE) results are reported in \cref{tab:regresssion} for the benchmark datasets under the selected parameters.  
\cref{tab:regresssion} indicates that the $\ell_{1/2}/\ell_2$ regularized RVFL model has superior performance across multiple benchmark datasets. $\ell_{1/2}/\ell_2$ achieves a superior balance between sparsity and predictive accuracy, substantially improving the robustness and practical applicability of RVFL model, and thereby validating its effectiveness in structured sparse learning.
\begin{table}[!htbp]
	\centering
	\caption{Mean squared error (MSE) on the test set for each benchmark dataset, with regularization parameters selected via three-fold cross-validation.}
	\label{tab:regresssion}
	\resizebox{\linewidth}{!}{
		\begin{tabular}{cccccccccccc} 
			\toprule
			Datasets                                                                        & RVFL-$\ell_2$       & RVFL-$\ell_{1}$         & \begin{tabular}[c]{@{}c@{}}RVFL-$\ell_{1}-\ell_{2}$ \\ DCA\end{tabular}  & IRLS$\ell_{p}$     & \begin{tabular}[c]{@{}c@{}}RVFL-$\ell_{1}-\ell_{2}$ \\ ADMM\end{tabular} & \begin{tabular}[c]{@{}c@{}}RVFL-$\ell_{1}-\ell_{2}$ \\ ADMMweighted\end{tabular} & \begin{tabular}[c]{@{}c@{}}RVFL-$\ell_{1}/\ell_{\infty}$ \\ ADMM\end{tabular} & \begin{tabular}[c]{@{}c@{}}RVFL-$\ell_{1}/\ell_{\infty}$ \\ FISTA\end{tabular} & RVFL-$\ell_{p}$         & RVFL-$\ell_{1}/\ell_{2}$            & RVFL-$\ell_{1/2}/\ell_{2}$              \\ 
			\hline
			\begin{tabular}[c]{@{}c@{}}HappinessRank\_2015\\(158,10)\end{tabular}           & 1.2060E-04 & 3.1146E-05 & 1.2008E-05 & 6.8062E-05 & 1.2278E-05 & 1.1598E-05         & 1.2007E-05   & 1.2008E-05    & 1.2007E-05 & 1.2007E-05       & \textbf{1.1388E-05}  \\
			\begin{tabular}[c]{@{}c@{}}fri\_c3\_100\_25\\(100,25)\end{tabular}              & 0.1671~    & 0.1903~    & 0.1341~    & 0.2573~    & 0.1367~    & 0.1342~            & 0.1341~      & 0.1341~       & 0.1341~    & 0.1341~          & \textbf{0.1339~}     \\
			\begin{tabular}[c]{@{}c@{}}cpu\\(209,7)\end{tabular}                            & 0.0064~    & 0.0084~    & 0.0019~    & 0.0016~    & 0.0018~    & 0.0019~            & 0.0019~      & 0.0019~       & 0.0019~    & 0.0019~          & \textbf{0.0014~}     \\
			\begin{tabular}[c]{@{}c@{}}bladder0\\(397,4)\end{tabular}                       & 4.3401~    & 3.9167~    & 5.3879~    & 95.1829~   & 4.0943~    & 5.1280~            & 5.3879~      & 5.3879~       & 5.3881~    & 5.3879~          & \textbf{2.0788~}     \\
			\begin{tabular}[c]{@{}c@{}}lung\\(167,9)\end{tabular}                           & 0.8747~    & 1.0447~    & 0.7428~    & 0.8540~    & 0.7575~    & 0.8515~            & 0.7428~      & 0.7428~       & 0.7428~    & 0.7428~          & \textbf{0.7417~}     \\
			\begin{tabular}[c]{@{}c@{}}csb\_ch4a\\(51,2)\end{tabular}                       & 0.1140~    & 0.1990~    & 0.0544~    & 0.7642~    & 0.0539~    & 0.0542~            & 0.0544~      & 0.0544~       & 0.0544~    & 0.0544~          & \textbf{0.0524~}     \\
			\begin{tabular}[c]{@{}c@{}}hipel\_mcleod\_askewaskew21\\(98,5)\end{tabular}     & 1.5026~    & 0.8092~    & 0.6919~    & 1.9863~    & 1.6558~    & \textbf{0.5955~}   & 0.6919~      & 0.6919~       & 0.6967~    & 0.6919~          & 0.6844~              \\
			\begin{tabular}[c]{@{}c@{}}hipel\_mcleod\_monthlywoods1\\(50,11)\end{tabular}   & 0.0028~    & 0.0034~    & 0.0019~    & 0.0027~    & 0.0019~    & 0.0019~            & 0.0019~      & 0.0019~       & 0.0019~    & 0.0019~          & \textbf{0.0019~}     \\
			\begin{tabular}[c]{@{}c@{}}Vinnie\\(380,2)\end{tabular}                         & 0.0421~    & 0.0421~    & 0.0421~    & 0.0421~    & 0.0421~    & 0.0421~            & 0.0421~      & 0.0421~       & 0.0421~    & 0.0421~          & \textbf{0.0421~}     \\
			\begin{tabular}[c]{@{}c@{}}alr161b\\(62,1)\end{tabular}                         & 0.0525~    & 0.0519~    & 0.0497~    & 0.0498~    & 0.0497~    & 0.0497~            & 0.0497~      & 0.0497~       & 0.0497~    & \textbf{0.0497~} & 0.0498~              \\
			\begin{tabular}[c]{@{}c@{}}hipel\_mcleod\_miscsimar41\\(200,3)\end{tabular}     & 0.0101~    & 0.0091~    & 0.0086~    & 0.0110~    & 0.0088~    & 0.0085~            & 0.0086~      & 0.0086~       & 0.0086~    & 0.0086~          & \textbf{0.0085~}     \\
			\begin{tabular}[c]{@{}c@{}}hipel\_mcleod\_noakesmckenzie1\\(100,5)\end{tabular} & 0.0863~    & 0.0795~    & 0.0670~    & 0.1692~    & 0.0676~    & 0.1045~            & 0.0670~      & 0.0670~       & 0.0670~    & 0.0670~          & \textbf{0.0669~}     \\
			\bottomrule
	\end{tabular}}
\end{table}

\section{Conclusions}\label{sec:conclusion}

In this work, we propose a novel nonconvex, nonsmooth, scale-invariant $\ell_{1/2}/\ell_{2}$ sparse optimization problem. Based on the eNSP, we first establish sufficient conditions that guarantee exact and stable sparse recovery for the associated constrained minimization problem. Then we design an appropriate splitting scheme and apply it into ADMM algorithm to solve the unconstrained $\ell_{1/2}/\ell_{2}$ minimization problem. Under reasonable assumptions and mild conditions, we prove the global convergence of the proposed ADMM algorithm for such a nonconvex, nonsmooth, and non-Lipschitz problem. Moreover, by leveraging the K\L{} property and K\L{} inequality, we analyze its linear convergence rate in specific cases. Finally,  numerical experiments demonstrate that $\ell_{1/2}/\ell_{2}$ consistently outperforms state-of-the-art approaches in both noiseless and noisy scenarios, with particularly superior performance on oversampled DCT matrices. Especially, the experiment on neural network sparsity and generalization further confirm the effectiveness of $\ell_{1/2}/\ell_{2}$ in  its predict accuracy. Future research will focus on a deeper investigation of the convergence behavior of $\text{ADMM}_{inner}$ and on the development of more effective optimization schemes to improve both computational efficiency and recovery performance.

\bibliographystyle{siamplain}
\bibliography{references}
\end{document}